\documentclass[final,journal,twoside,twocolumn,10pt]{IEEEtran}
\usepackage{acronym}
\usepackage{algorithm}
\usepackage{algpseudocode}
\usepackage{amsfonts}
\usepackage{amsmath}
\usepackage{amsthm}
\usepackage{graphicx}
\usepackage{hyperref}
\usepackage{multirow}
\usepackage{upgreek}
\usepackage{xcolor}

\newcommand{\bracket}[1]{\begin{bmatrix} #1 \end{bmatrix}}

\newcommand{\insertpdf}[1]{\includegraphics[scale=0.8]{#1}}
\newcommand{\mycaption}[2]{\caption[#1]{#1 #2}}
\newtheorem{theorem}{Theorem}

\newcommand{\ft}{\mathrm{ft}}
\newcommand{\m}{\mathrm{m}}
\newcommand{\SSE}{\mathit{SSE}}
\newcommand{\T}{\mathrm{T}}

\DeclareMathOperator{\sort}{sort}
\DeclareMathOperator{\tr}{tr}

\acrodef{1D}[1D]{one-dimensional}
\acrodef{2D}[2D]{two-dimensional}
\acrodef{3D}[3D]{three-dimensional}
\acrodef{API}{application programming interface}
\acrodef{BSX}{binary singleton-expansion}
\acrodef{CPD}{canonical polyadic decomposition}
\acrodef{CPU}{central processing unit}
\acrodef{CSC}{compressed sparse column}
\acrodef{GB}{Gr\"obner basis}
\acrodef{GPU}{graphics processing unit}
\acrodef{IDE}{integrated development environment}
\acrodef{LU}{lower- and upper-triangular}
\acrodef{MDA}{multidimensional array}
\acrodef{MEX}{MATLAB executable}
\acrodef{MV}{matrix-vector}
\acrodef{R2024}{Release~2024}
\acrodef{R2026}{Release~2026}
\acrodef{RREF}{row-reduced echelon form}
\acrodef{RRLU}{rank-revealing \ac{LU}}
\acrodef{QR}{unitary-triangular}
\acrodef{Qr}{any-degree unitary-triangular}
\acrodef{RT}{Ricci-notation tensor}
\acrodef{sci/tech}{scientific/technical}
\acrodef{SSE}{sum squared error}
\acrodefindefinite{SSE}{an}{a}
\acrodef{SVD}{singular value decomposition}
\acrodef{USX}{unary singleton-expansion}

\begin{document}


\title{Ricci-Notation Tensor Framework for \\Numerical Algebraic Geometry via Any-Degree Unitary-Triangular Factorization}

\author{Dileepan~Joseph (dil.joseph@ualberta.ca)\\
\vspace{1ex}
Department of Electrical and Computer Engineering\\
University of Alberta, Edmonton, AB, Canada\\
\vspace{1ex}
\today}

\date{}
\maketitle


\begin{abstract}
The unitary-triangular (QR) factorization of linear algebra may be used to robustly and efficiently solve a linear system. Toward a comparable numerical method to solve a polynomial system of higher degree, this paper proposes an \emph{any-degree} unitary-triangular (Qr) factorization, which for a degree-one system reduces to the QR factorization. The work develops a tensor framework, i.e., codesigned algebra and software, where polynomial system coefficients are represented by a vector-shaped sparse tensor, a multidimensional array whose number of Ricci-notation indices, called the tensor degree, equals the highest monomial degree of the system. With the proposed Qr factorization, the coefficient tensor decomposes into a product of unitary and triangular factors that, in general, also have Ricci-notation indices and sparse entries. The unitary factor defines a unitary transform, a generalization of the related linear algebra concept to tensor algebra, that can triangularize a polynomial system while preserving its solution set, whether zero- or positive-dimensional. The work extends the author’s Ricci-notation tensor framework, providing new algebra and new software to model, construct, and factorize polynomial systems in this manner. After applying the approach to numerically triangularize two zero-dimensional systems, chosen for educational value, results are compared to the Gr\"obner-basis (GB) method for triangularizing polynomial systems symbolically. One problem is of degree three, with three equations and unknowns, and the other of degree two, with four equations and unknowns. Although it resembles GB triangularization, the proposed Qr factorization has a completely different pedigree associated with numerical methods.
\end{abstract}


\section{Introduction}
\label{sec:introduction}

Algebraic geometry concerns the algebraic set, solution set, or zero set of a given polynomial system of equations. Unlike symbolic frameworks, numeric ones embrace approximations. Homotopy continuation, a popular numerical approach, constructs a reference system having known zeros and related to the given system. The reference and its zeros map, following a real parameter going from a start to an end value, to the given system and its zeros. In detailed literature reviews, Hauenstein and Sommese \cite{Hauenstein2017a} and Hauenstein and Wampler \cite{Hauenstein2017b} elaborate on key developments like witness sets that represent all zeros via intersections to specified linear spaces. Continuation can handle positive-dimensional zero sets, e.g., surfaces, curves, and points, in addition to zero-dimensional ones, i.e., points alone. In the context of robot mechanisms, stability analysis, and model selection, Wampler and Plecnik \cite{Wampler2025}, Menini \emph{et al.} \cite{Menini2023}, and Gross \emph{et al.} \cite{Gross2016} elaborate on continuation.

Another approach to numerical algebraic geometry, which Vanderstukken and De~Lathauwer \cite{Vanderstukken2021a} review and improve, features tensor factorization. Polynomial system coefficients map to nonzero entries of a large and sparse Macaulay matrix. Its null space yields a basis that contains the zero set of the given system. The approach converts the basis to a tensor and applies \ac{CPD}, a generalization of \ac{SVD} from \ac{MV} to tensor algebra, along with post-processing to compute and extract all valid zeros as vectors. Ishteva and Dreesen \cite{Ishteva2022} focus on degree-two systems and a partly-symmetric tensor model, yielding a paper of educational value that follows a similar approach. Vanderstukken \emph{et al.} \cite{Vanderstukken2021b} generalize the Macaulay-\ac{CPD} approach beyond zero-dimensional solution sets.

Reviews and developments concerning \ac{CPD} inform numerical algebraic geometry even when motivated by different objectives. Addressing model selection of dynamical systems from measured data, Batselier \cite{Batselier2022} surveys tensor factorization literature, emphasizing \ac{CPD} over Tucker and tensor train decompositions because \ac{CPD} exhibits a rank revealing property. In the context of data analysis and a generalization of principal component analysis, Larsen and Kolda \cite{Larsen2022} accelerate the \ac{CPD} of large tensors, with relatively little loss of accuracy, by sampling nonzeros randomly and systematically.

Unlike the numerical approaches, the symbolic approach to solve a polynomial system yields a triangular system having the same solution set, whether zero- or positive-dimensional, along with other useful properties. Buchberger and Kauers \cite{Buchberger2010} summarize this reduced \ac{GB} approach, focusing for educational purposes on simple examples, some from \ac{sci/tech} applications like modular robotics. Faug\`ere and Mou \cite{Faugere2017} review \ac{GB} algorithms, developed over decades, before elaborating on a new way, for the zero-dimensional case, to convert an intermediate representation to the elimination-order one, called triangular here, such conversion being a bottleneck notable in cryptanalysis.

Reduced \ac{GB} implementations, like \texttt{gbasis} of MATLAB's Symbolic Math Toolbox \cite{MathWorks2025s}, allow polynomial systems with symbolic coefficients. After triangularization, one may enter numeric values into coefficients before numerically solving a sequence of univariate equations, one variable at a time, to obtain one or more solutions in the zero set. To solve a problem in modular robotics, Liu and Han \cite{Liu2021} compare continuation and \ac{GB} approaches. While the numerical approach was easier to apply, the symbolic-numeric hybrid computed faster with equivalent solutions for the linkages involved. Elumalai \emph{et al.} \cite{Elumalai2019} and Yang \emph{et al.} \cite{Yang2015} compare numerical optimization to hybrid optimization for problems, in geoscience and power electronics, respectively, definable by polynomial systems. Both works favoured the hybrid approach. Triangularization, via a reduced \ac{GB}, enabled global optimization.

This work introduces \iac{RT} framework to formulate and triangularize polynomial systems numerically. It proposes the \ac{Qr} factorization for numerical algebraic geometry, a tensor factorization that simplifies for linear systems to \ac{QR} factorization. The work thereby generalizes the unitary transforms of linear algebra. Round-off error aside, proposed transforms guarantee solution set invariance, whether zero- or positive-dimensional, regardless of system degree. Coefficients before and after triangularization, and the unitary factors that define the unitary transforms, are all sparse tensors associated with a certain number, called the degree, of \ac{RT} indices.

In stages, Harrison \cite{Harrison2016a}, Harrison and Joseph \cite{Harrison2016b, Harrison2018}, and Joseph \cite{Joseph2024b} have developed, alongside \ac{RT} software, a dual-variant index notation with Einstein summation convention, the \ac{RT} algebra, for \emph{numeric} tensor purposes. This work develops the \ac{RT} framework for algebraic geometry, i.e., for \emph{geometric} tensor purposes. At present, the framework concerns neither the Ricci calculus nor differential geometry.

Other researchers who write about \ac{QR} factorization in the context of tensors and/or polynomial systems do so in very different ways from this work. Leng \emph{et al.} \cite{Leng2025} develop a block-based variation of \ac{QR} factorization to accelerate \ac{MV} operations that feature in deep learning when performed on \acp{GPU} with tensor cores. Similarly, Rohrig-Zollner \emph{et al.} \cite{RohrigZollner2022} develop \ac{QR} factorization variants to improve tensor trains decompositions when distributed over \ac{CPU} clusters. Meanwhile, for Tucker and truncated Tucker tensor decompositions, Kaloorazi \emph{et al.} \cite{Kaloorazi2024} and Beaupere \emph{et al.} \cite{Beaupere2023} develop \ac{QR}-based alternatives to \ac{SVD}-based rank reductions of matrices derived from large \acp{MDA}. For an overdetermined polynomial system, Govindarajan \emph{et al.} \cite{Govindarajan2024} compare \ac{QR} factorization with column pivoting and \ac{LU} factorization with full pivoting, selecting the latter, to compute the null space of a Macaulay matrix.

The proposed \ac{RT} framework includes tensor software to realize, with programmatic and computational efficiency in MATLAB on \acp{CPU}, sparse models expressed with the \ac{RT} algebra. Others have codesigned software with tensor algebra to represent and manipulate sparse tensors in MATLAB. For \ac{CPD} and other \ac{SVD}-like decompositions expressed with an $n$-mode notation, Bader and Kolda \cite{Bader2007} and Sorber \emph{et al.} \cite{Sorber2015} contributed the Tensor Toolbox and Tensorlab, respectively. In their GenTen library for \ac{CPD}, Phipps and Kolda \cite{Phipps2019} exploited parallel computing while revisiting a key Tensor Toolbox function for matricized tensor times Khatri-Rao products.

Recent research on the acceleration of sparse tensor operations spans high- to low-level approaches. Bik \emph{et al.} \cite{Bik2022} develop a compiler to abstract away details while supporting a variety of index and nonzero value formats. The best choice would be architecture dependent. Scheffler \emph{et al.} \cite{Scheffler2023} develop \iac{CPU} architecture with special registers and assembly language instructions to handle the streaming by index of nonzero values from cache and other memory into arithmetic units.

Tensor software of the \ac{RT} framework means the RTToolbox \cite{Joseph2024b}, featuring \texttt{tensor} and \texttt{index} classes, plus MATLAB. Codesigned with \ac{RT} algebraic geometry, this work extends the RTToolbox to support sparse tensors, replacing \ac{R2024} \cite{Joseph2024a} with the accompanying \ac{R2026}. Where possible, the design leverages vectorized operations and compiled functions of the MATLAB kernel. Extensions build on sparse \ac{MDA} representation as \ac{1D}, sometimes \ac{2D}, kernel sparse arrays.

The \ac{RT} framework realizes a sparse tensor by constructing a \texttt{tensor} object using a sparse \ac{MDA} object, defined by a \texttt{sparse1} class developed for the RTToolbox, and an \texttt{index} object. New functions and methods for the toolbox and its \texttt{tensor} class support sparsity and \ac{Qr} factorization. These include \texttt{spmex1}, \iac{MEX} function. Compilable upon RTToolbox installation, \texttt{spmex1} accelerates null-space and \ac{BSX} operations on sparse \ac{2D} arrays in the \ac{CSC} format MATLAB uses. A mid-level approach to acceleration, the function exploits a C language \ac{API} \cite{MathWorks2025c} to compute a precursor to a null space, \iac{RREF}, via \iac{RRLU} factorization. Not computable via MATLAB's \texttt{lu} function, the approach outperforms MATLAB alternatives. Compute-bound \ac{BSX} operations generalize Khatri-Rao products and replace memory-bound alternatives possible via MATLAB \ac{USX} operations.

Concept demos included with the RTToolbox, \ac{R2026}, offer examples chosen for educational value. These include a box dimensions problem, understandable even by high school students, that may be triangularized via \ac{Qr} factorization without internal iteration. A second demo, requiring iteration, applies the \ac{Qr} factorization to a two-segment robotic arm problem, understandable by university \ac{sci/tech} students. Both demos compare to a reduced \ac{GB} approach using \texttt{gbasis}.

The rest of this paper has four sections. Section~\ref{sec:tensor algebra} presents the tensor algebra this work adds to the \ac{RT} framework, in particular a model and method to represent and triangularize a polynomial system. Section~\ref{sec:concept demos} presents two concept demos, called the box dimensions and two-segment arm problems, respectively. Section~\ref{sec:tensor software} presents the RTToolbox, \ac{R2026}, emphasizing contributions of this second release in comparison to the initial one, \ac{R2024}. With the benefit of details from the preceding sections, Section~\ref{sec:conclusions} summarizes this work.

\acused{3D}
\acused{IDE}
\acused{SSE}


\section{Tensor Algebra}
\label{sec:tensor algebra}

This section introduces \ac{RT} algebraic geometry, in particular unitary and triangular identities that define factors of \iac{Qr} factorization, as well as canonical and other representations to specify polynomial systems and prove zero-set invariance upon unitary transformation. The section also presents a numerical method to compute the unitary and triangular factors.

\subsection{Form \& Factorization}

Table~\ref{tab:equations} categorizes systems of equations along two axes. One axis divides systems into linear, polynomial, and nonlinear. The other axis divides systems by the number of variables. Although linear system models are special cases of polynomial and nonlinear system models, specialized solution methods exist for the former. Explanations of these methods, e.g., in university \ac{sci/tech} courses, arguably benefit from the use of specialized models. Theoretically, one could model a polynomial system with a nonlinear system model.

\begin{table}[t]
\centering
\mycaption{Models and methods to solve systems of equations.}{For multivariate polynomial systems, this work introduces a \acf{RT} model and an \acf{Qr} factorization method. The work builds on published \acs{RT} algebra fundamentals \cite{Joseph2024b}. Table adapted from Harrison's Ph.D. thesis \cite{Harrison2016a}.}
\label{tab:equations}
\vspace{1ex}
\begin{tabular}{|c|c|c|}
\hline
Type & Univariate & Multivariate \\
\hline
\hline
\multirow{2}{*}{Linear} & $a x = b$ & $\mathbf{A} \mathbf{x} = \mathbf{b}$ \\
& Scalar division & LU factorization, etc. \\
\hline
\multirow{2}{*}{Polynomial} & $a_n x^n + \cdots a_1 x + a_0 = 0$ & $\mathbf{a}_\mathbf{j} \Pi^\mathbf{j}(\mathbf{x}, 1) = \mathbf{0}$ \\
& Companion matrix, etc. & \acs{Qr} factorization, etc. \\
\hline
\multirow{2}{*}{Nonlinear} & $f(x) = 0$ & $\mathbf{f}(\mathbf{x}) = \mathbf{0}$ \\
& Brent's method, etc. & Newton's method, etc. \\
\hline
\end{tabular}
\end{table}

In proposing the \ac{Qr} factorization method to triangularize a multivariate polynomial system, this work also proposes and manipulates models using the \ac{RT} algebra, a tensor algebra having a dual-variant index notation. Consider a polynomial system of $M$ equations and $N$ scalar variables, with a highest monomial degree of $D$. This work models such a system using one coefficient tensor, $\mathbf{a}$, and one unknown variable, $\mathbf{x}$, shaped as $M \times 1$ and $N \times 1$ vectors, in a canonical form:
\begin{align}
\mathbf{a}_\mathbf{j} \Pi^\mathbf{j}(\mathbf{x}, 1) &= \mathbf{0} \text{.}
\end{align}

Along with row/column dimension sizes, the coefficient tensor, $\mathbf{a}$, has $D$ subscript or true-variant indices, each of $N + 1$ dimension size, expressed as a vector index, $\mathbf{j}$. The \ac{RT} algebra distinguishes true/false-variant indices from row/column indices and calls such a tensor a degree-$D$ column vector, making the unknown variable, $\mathbf{x}$, a degree-zero column vector. The function, $\Pi(\mathbf{x}, 1)$, models a scalar, $\Pi$, with $D$ superscript or false-variant indices. Where possible, inline math omits tensor indices and context implies dimension sizes.

With one argument, the degree-$D$ scalar, $\Pi$, models an outer product of the argument with itself, $D$ times, after reshaping from an $N \times 1$ degree-zero form, $\mathbf{x}$, to a $1 \times 1$ degree-one form, $x$, using a term, reshape, found in \ac{MDA} software:
\begin{align}
\Pi^\mathbf{j}(\mathbf{x}) &= x^{j_1} x^{j_2} \cdots x^{j_D} \text{.}
\end{align}
As shown, the algebra models outer products via differing indices. An overloaded two-argument form of what this work calls the monomial-values scalar, $\Pi$, equals the one-argument form where the given arguments concatenate vertically:
\begin{align}
\Pi^\mathbf{j}(\mathbf{x}, \mathbf{y}) &= \Pi^\mathbf{j} \left( \bracket{\mathbf{x} \\ \mathbf{y}} \right) \text{.}
\end{align}

Generalizing the monomial-values scalar, $\Pi$, one can model a system, $\mathbf{a} \Pi$, with coefficients, $\mathbf{a}$, and multiple scalar variables or unknowns, $x$, $y$, $z$, etc., in canonical form easily:
\begin{align}
\mathbf{a}_\mathbf{j} \Pi^\mathbf{j}(x, y, z, 1) &= \mathbf{0} \text{.}
\end{align}
Repeated indices of mixed variant, subscript and superscript, express inner products, i.e., summations, along those indices. Repeated indices of the same variants, also allowed in the \ac{RT} algebra, express entrywise products along those indices.

The canonical form of a system, $\mathbf{a} \Pi$, in multiple unknowns, $x_1$, $x_2$, etc., does not prevent summation of equivalent monomials. However, any most-sparse representation does require asymmetry, e.g., $2 x_1 x_2$ instead of $x_1 x_2 + x_2 x_1$. A product, $\mathbf{a} \varsigma$, serves to model one most-sparse representation:
\begin{align}
\mathbf{a}_\mathbf{j} &= \mathbf{a}_\mathbf{i} \varsigma_\mathbf{j}^\mathbf{i} \text{.}
\end{align}
Added to the \ac{RT} algebra, the variant sigma symbol, $\varsigma$, is a mixed-variant scalar involving a $\sort$ function to sort index values in ascending order (largest value to the right), an order that does matter when modeling upper triangularity:
\begin{align}
\varsigma_\mathbf{j}^\mathbf{i} &\equiv
\begin{cases}
1, & \mathbf{j} = \sort(\mathbf{i}) \text{,} \\
0, & \text{otherwise} \text{.}
\end{cases}
\end{align}

The \ac{RT} algebra models upper-triangular coefficients, $\mathbf{r}$, using degree-one column and row vectors, $\mathbf{e}$ and $\mathbf{e}^\T$, basis operators that define reshape operations, representable as sparse multiplication, between row/column and tensor indices:
\begin{align}
\mathbf{r}_\mathbf{k} &= \mathbf{e}^m \mathbf{e}_n^\T \mathbf{r}_\mathbf{j} \varsigma_{m\mathbf{k}}^{n\mathbf{j}} \text{.}
\end{align}
Nonzero entries, always one, of either operator, $\mathbf{e}$ or $\mathbf{e}^\T$, occur if and only if the row or column index value equals the tensor index value. A system of $M$ equations and $N$ unknowns has an $M \times 1$ coefficient vector, $\mathbf{r}$, so basis operators have row/column and index values, $m$ and $n$, that range from $1$ to $M$ here. The model applies to systems, $\mathbf{r} \Pi$, whether or not $M$ equals $N$. Indices in the index vectors, $\mathbf{k}$ and $\mathbf{j}$, each range from $1$ to $N + 1$, given augmentation of unknowns, $\mathbf{x}$, by one.

In the case of a linear system, the \ac{Qr} factorization simplifies to the \ac{QR} factorization of linear algebra. This degree-one case, upper-triangular by convention, clarifies the end augmentation of unknowns and ascending sort of indices in the any-degree case. Converting single tensor indices of vector-shaped tensors to column indices yields canonical and triangular forms:
\begin{align}
\bracket{\mathbf{A} & -\mathbf{b}} \bracket{\mathbf{x} \\ 1} &= \mathbf{0} \text{,} \\
\bracket{\mathbf{R} & \mathbf{r}} \bracket{\mathbf{x} \\ 1} &= \mathbf{0} \text{,}
\end{align}
where, via \ac{QR} factorization of an $M \times (N + 1)$ matrix,
\begin{align}
\bracket{\mathbf{R} & \mathbf{r}} &= \mathbf{Q}^* \bracket{\mathbf{A} & -\mathbf{b}} \text{,} \\
\bracket{\mathbf{A} & -\mathbf{b}} &= \mathbf{Q} \bracket{\mathbf{R} & \mathbf{r}} \text{.}
\end{align}
These models indicate two unitary transforms, $\mathbf{Q}^*$ and $\mathbf{Q}$, where the one with a conjugated factor effects the triangularization of interest, an observation that generalizes.

Consider a degree-$D$ system, $\mathbf{a} \Pi$, of $M$ equations and $N$ unknowns with equivalent monomials collected ($\mathbf{a}$ equals $\mathbf{a} \varsigma$). This work introduces a degree-$(D + \Delta D)$ triangular system, $\mathbf{r} \Pi$, derived by unitary transformation via a degree-$\Delta D$ unitary factor, $\mathbf{Q}_\text{u}$, as defined by the proposed \ac{Qr} factorization:
\begin{align}
\mathbf{r}_\mathbf{k} &= (\mathbf{Q}_\text{u}^\mathbf{i})^* \mathbf{a}_\mathbf{j} \varsigma_\mathbf{k}^{\mathbf{i}\mathbf{j}} \text{,} \\
\mathbf{a}_\mathbf{j} &= \mathbf{Q}_\mathbf{h}^\text{v} \mathbf{r}_\mathbf{k} \varsigma_{\mathbf{j}\mathbf{0}}^{\mathbf{h}\mathbf{k}} \text{.}
\end{align}
The model involves three factors: two $M \times M$ matrices, $\mathbf{Q}_\text{u}$ and $\mathbf{Q}^\text{v}$, of degree $\Delta D$; and an $M \times 1$ upper-triangular vector, $\mathbf{r}$, of degree $D + \Delta D$. In the \ac{RT} algebra, indices enclosed in a (conjugate) transposition change variants. Forward and backward unitary factors, $\mathbf{Q}_\text{u}$ and $\mathbf{Q}^\text{v}$, equate to each other in the degree-one case, where one symbol, $\mathbf{Q}$, suffices.

To simplify notation related to triangular coefficients, this work introduces a shorthand where the zero index indicates an end value specified usually by context, e.g., $N + 1$, and the zero vector index indicates a vector of end values, i.e.,
\begin{align}
\mathbf{r}_\mathbf{k} \varsigma_{\mathbf{j}\underline{\mathbf{0}}}^{\mathbf{h}\mathbf{k}} &\equiv \mathbf{r}_\mathbf{k} \varsigma_{\mathbf{j}\underline{\mathbf{1}(N + 1)}}^{\mathbf{h}\mathbf{k}} \text{.}
\end{align}

Unitary factors, $\mathbf{Q}_\text{u}$ and $\mathbf{Q}^\text{v}$, have $\Delta D$ tensor indices. This work treats the number, the incremental degree, as a parameter of the \ac{Qr} factorization that a user specifies to triangularize a given degree-$D$ system. Regardless of incremental degree and triangularization, unitary factors must obey two identities:
\begin{align}
(\mathbf{Q}_\text{u}^\mathbf{0})^* \mathbf{Q}_\text{u}^\mathbf{0} &= \mathbf{I} \text{,} \\
\mathbf{Q}_\mathbf{i}^\text{v} (\mathbf{Q}_\text{u}^\mathbf{j})^* \varsigma_\mathbf{k}^{\mathbf{i}\mathbf{j}} &= \mathbf{I} \delta_\mathbf{k}^\mathbf{0} \text{.}
\end{align}

The \ac{RT} algebra, as previously defined, includes a Dirac delta symbol, $\delta$, to express tensor index equivalence. Like the proposed variant sigma symbol, $\varsigma$, to collect equivalent system monomials, one may define it simply as a sparse tensor:
\begin{align}
\delta_\mathbf{j}^\mathbf{i} &\equiv
\begin{cases}
1, & \mathbf{j} = \mathbf{i} \text{,} \\
0, & \text{otherwise} \text{.}
\end{cases}
\end{align}

In the special case of zero incremental degree, a value that suffices to triangularize any linear system given the linear-algebra \ac{QR} theory, the two unitary identities simplify:
\begin{align}
(\mathbf{Q}_\text{u})^* \mathbf{Q}_\text{u} &= \mathbf{I} \text{,} \\
\mathbf{Q}^\text{v} (\mathbf{Q}_\text{u})^* &= \mathbf{I} \text{.}
\end{align}
Each factor must therefore equal the other. For the general case of nonzero incremental degree, a related observation applies. Consider an end-index slice of the second identity alone:
\begin{align}
\mathbf{Q}_\mathbf{i}^\text{v} (\mathbf{Q}_\text{u}^\mathbf{j})^* \varsigma_\mathbf{0}^{\mathbf{i}\mathbf{j}} &= \mathbf{I} \delta_\mathbf{0}^\mathbf{0} \text{.}
\end{align}
Simplifying yields an end-index identity, true for any degree, where an end-index slice of either factor equals the linear-algebra inverse of the conjugate transpose of the other:
\begin{align}
\mathbf{Q}_\mathbf{0}^\text{v} (\mathbf{Q}_\text{u}^\mathbf{0})^* &= \mathbf{I} \text{.}
\end{align}
Actually, considering the first unitary identity also, the end-index slice of each factor must always equal the other.

\subsection{Zero-Set Invariance}

Consider the zero set, $\mathcal{X}$, of a polynomial system, $\mathbf{a} \Pi$, having $M$ equations and $N$ unknowns. Each member of the set, $\mathbf{x} \in \mathcal{X}$, solves the system and no other solutions exist. Coefficients, $\mathbf{a}$, play a role in determining set size, related to rank in linear algebra. Simple relations, $M > N$, $M = N$, and $M < N$, do not define set size, whether empty, finite (i.e., zero-dimensional), or infinite (i.e., positive-dimensional).

Assuming a successful \ac{Qr} factorization of given system coefficients, $\mathbf{a}$, ideally the zero sets of given and triangular systems, $\mathbf{a} \Pi$ and $\mathbf{r} \Pi$, equate. Two theorems claim that a zero of the given system implies a zero of the triangular system and vice versa. Whereas the \ac{Qr} factorization adds a triangularity requirement, the theorem proofs depend only on properties of the unitary factors, $\mathbf{Q}_\text{u}$ and $\mathbf{Q}^\text{v}$, which define a pair of unitary transforms that produce either system from the other.

\begin{theorem}[Forward invariance theorem]
The zero set of a triangular system, $\mathbf{r} \Pi$, contains the zero set of a given system, $\mathbf{a} \Pi$, where coefficients transform via a forward unitary factor, $\mathbf{Q}_\text{u}$, obtained by \ac{Qr} factorization of given coefficients, $\mathbf{a}$:
\begin{align*}
\mathbf{r}_\mathbf{k} &= (\mathbf{Q}_\text{u}^\mathbf{i})^* \mathbf{a}_\mathbf{j} \varsigma_\mathbf{k}^{\mathbf{i}\mathbf{j}} \text{.}
\end{align*}
\end{theorem}

\begin{proof}
Consider a degree-$D$ system, $\mathbf{a} \Pi$, that in canonical form equates a known $M \times 1$ vector function of an unknown $N \times 1$ vector, $\mathbf{x}$, to zero. Premultiplying both canonical sides by the same $M \times M$ matrix, $\mathbf{U}$, results in a transformed system, $\mathbf{U} \mathbf{a} \Pi$, whose zero set always contains the given one:
\begin{align}
\mathbf{U} \mathbf{a}_\mathbf{j} \Pi^\mathbf{j}(\mathbf{x}, 1) &= \mathbf{0} \text{.}
\end{align}

With a forward degree-$\Delta D$ factor, $\mathbf{Q}_\text{u}$, of given coefficients, $\mathbf{a}$, construct the premultiplier matrix, $\mathbf{U}$, as a forward transform, $\mathbf{Q}_\text{u}^* \Pi$, that depends polynomially on the vector, $\mathbf{x}$:
\begin{align}
\mathbf{U} &= (\mathbf{Q}_\text{u}^\mathbf{i})^* \Pi^\mathbf{i}(\mathbf{x}, 1) \text{.}
\end{align}
Put the transformed system, $\mathbf{U} \mathbf{a} \Pi$, into the canonical form of an equal system, $\mathbf{b} \Pi$, identifying a coefficient vector, $\mathbf{b}$:
\begin{align}
\mathbf{b}_{\mathbf{i}\mathbf{j}} \Pi^{\mathbf{i}\mathbf{j}}(\mathbf{x}, 1) &= \mathbf{0} \text{,} \\
\mathbf{b}_{\mathbf{i}\mathbf{j}} &= (\mathbf{Q}_\text{u}^\mathbf{i})^* \mathbf{a}_\mathbf{j} \text{,}
\end{align}
where, according to the monomial-values scalar function,
\begin{align}
\Pi^\mathbf{i}(\mathbf{x}, 1) \Pi^\mathbf{j}(\mathbf{x}, 1) &= \Pi^{\mathbf{i}\mathbf{j}}(\mathbf{x}, 1) \text{.}
\end{align}

The vector, $\mathbf{b}$, that defines the transformed system, $\mathbf{b} \Pi$, resembles the degree-$(D + \Delta D)$ vector, $\mathbf{r}$, that defines the triangular system, $\mathbf{r} \Pi$. The vectors equate upon collection of equivalent monomials, i.e., $\mathbf{r}$ equals $\mathbf{b} \varsigma$, an operation that does not alter the associated zero set. Thus, every zero of the given system, $\mathbf{a}$, implies a zero of the triangular system, $\mathbf{r}$.
\end{proof}

\begin{theorem}[Backward invariance theorem]
The zero set of a given system, $\mathbf{a} \Pi$, contains the zero set of a triangular system, $\mathbf{r} \Pi$, where coefficients transform via a backward unitary factor, $\mathbf{Q}^\text{v}$, obtained by \ac{Qr} factorization of given coefficients, $\mathbf{a}$:
\begin{align*}
\mathbf{a}_\mathbf{j} &= \mathbf{Q}_\mathbf{h}^\text{v} \mathbf{r}_\mathbf{k} \varsigma_{\mathbf{j}\mathbf{0}}^{\mathbf{h}\mathbf{k}} \text{.}
\end{align*}
\end{theorem}

\begin{proof}
To prove that every zero of the triangular system, $\mathbf{r} \Pi$, implies a zero of the given system, $\mathbf{a} \Pi$, start by replacing the triangular system coefficients, $\mathbf{r}$, by their values, $\mathbf{Q}_\text{u}^* \mathbf{a} \varsigma$:
\begin{align}
(\mathbf{Q}_\text{u}^\mathbf{i})^* \mathbf{a}_\mathbf{j} \varsigma_\mathbf{k}^{\mathbf{i}\mathbf{j}} \Pi^\mathbf{k}(\mathbf{x}, 1) &= \mathbf{0} \text{.}
\end{align}
Eliminate the collection of monomials to yield an uncollected system in canonical form, $\mathbf{Q}^* \mathbf{a} \Pi$, with identical zero set:
\begin{align}
(\mathbf{Q}_\text{u}^\mathbf{i})^* \mathbf{a}_\mathbf{j} \Pi^{\mathbf{i}\mathbf{j}}(\mathbf{x}, 1) &= \mathbf{0} \text{.}
\end{align}

Premultiply both sides of the uncollected equation by the same $M \times M$ matrix, $\mathbf{V}$, to identify a new system, $\mathbf{V} \mathbf{Q}_\text{u}^* \mathbf{a} \Pi$, whose zero set contains that of the triangular system, $\mathbf{r} \Pi$:
\begin{align}
\mathbf{V} (\mathbf{Q}_\text{u}^\mathbf{i})^* \mathbf{a}_\mathbf{j} \Pi^{\mathbf{i}\mathbf{j}}(\mathbf{x}, 1) &= \mathbf{0} \text{.}
\end{align}
Construct a premultiplier, $\mathbf{V}$, called a backward transform, $\mathbf{Q}^\text{v} \Pi$, using  the unknown vector, $\mathbf{x}$, and the backward degree-$\Delta D$ factor, $\mathbf{Q}^\text{v}$, associated with the forward factor, $\mathbf{Q}_\text{u}$:
\begin{align}
\mathbf{V} &= \mathbf{Q}_\mathbf{h}^\text{v} \Pi^\mathbf{h}(\mathbf{x}, 1) \text{.}
\end{align}

Rearrange scalars of the premultiplied system, $\mathbf{V} \mathbf{Q}_\text{u}^* \mathbf{a} \Pi$, to factor out the given system, $\mathbf{a} \Pi$, in the canonical form:
\begin{align}
\mathbf{Q}_\mathbf{h}^\text{v} (\mathbf{Q}_\text{u}^\mathbf{i})^* \Pi^{\mathbf{h}\mathbf{i}}(\mathbf{x}, 1) \, \mathbf{a}_\mathbf{j} \Pi^\mathbf{j}(\mathbf{x}, 1) &= \mathbf{0} \text{,}
\end{align}
where, according to the monomial-values scalar function,
\begin{align}
\Pi^\mathbf{h}(\mathbf{x}, 1) \, \Pi^{\mathbf{i}\mathbf{j}}(\mathbf{x}, 1) &= \Pi^{\mathbf{h}\mathbf{i}}(\mathbf{x}, 1) \, \Pi^\mathbf{j}(\mathbf{x}, 1) \text{.}
\end{align}
Replace one scalar, $\Pi$, with an equivalent, $\varsigma \Pi$, to further factor out the left side, $\mathbf{Q}^\text{v} \mathbf{Q}_\text{u}^* \varsigma$, of the second unitary identity:
\begin{align}
\mathbf{Q}_\mathbf{h}^\text{v} (\mathbf{Q}_\text{u}^\mathbf{i})^* \varsigma_\mathbf{k}^{\mathbf{h}\mathbf{i}} \Pi^\mathbf{k}(\mathbf{x}, 1) \, \mathbf{a}_\mathbf{j} \Pi^\mathbf{j}(\mathbf{x}, 1) &= \mathbf{0} \text{.}
\end{align}

Via the second unitary identity, eliminate unitary factors, $\mathbf{Q}^\text{v}$ and $\mathbf{Q}_\text{u}$, in the rearranged premultiplied system, $\mathbf{V} \mathbf{Q}_\text{u}^* \mathbf{a} \Pi$, whose zero set contains that of the triangular system, $\mathbf{r} \Pi$:
\begin{align}
\mathbf{I} \delta_\mathbf{k}^\mathbf{0} \Pi^\mathbf{k}(\mathbf{x}, 1) \, \mathbf{a}_\mathbf{j} \Pi^\mathbf{j}(\mathbf{x}, 1) &= \mathbf{0} \text{.}
\end{align}
The scalar product term, $\delta^\mathbf{0} \Pi$, equates to an end slice, $\Pi^\mathbf{0}$, of the monomial-values scalar, which in turn equates to one:
\begin{align}
\Pi^\mathbf{0}(\mathbf{x}, 1) &= 1 \text{.}
\end{align}
Thus, after simplification, the rearranged premultiplied system yields the given system, implying every zero of the triangular system, $\mathbf{r} \Pi$, must be a zero of the given system, $\mathbf{a} \Pi$.
\end{proof}

The backward invariance theorem requires only the second, two-factor, unitary identity and the forward invariance theorem requires neither identity. For a zero incremental degree, $\Delta D$, the second identity directly implies the first, one-factor, identity. Although not employed by a proposed null-space method, one can show that any solution to the two-factor identity may be transformed, notwithstanding round-off, into a solution that satisfies both identities. Thus, the one-factor identity constrains degrees of freedom to favour numerical stability while the two-factor identity guarantees zero-set invariance.

\subsection{Null-Space Method}

Divisible into parts, Algorithm~\ref{fig:nullspace} computes the \ac{Qr} factorization of a system, subject to parameters. The first part non-iteratively computes initial unitary factors using a weighted null-space basis to ensure triangularization, unitary identities aside. To satisfy unitary identities, the second part defines and performs an iterative optimization. The last part non-iteratively completes the forward unitary factor, given the null-space basis and weights, and computes the triangular factor.

\begin{algorithm}[t]
\mycaption{Null-space method to obtain \acs{Qr} factors.}{Given real coefficients, $\mathbf{a}$, of a degree-$D$ system, $\mathbf{a} \Pi$, an incremental degree parameter, $\Delta D$, and optimization parameters, $\mathrm{options}$, this algorithm initializes, optimizes, and finalizes degree-$\Delta D$ unitary factors, $\mathbf{Q}_\text{u}$ and $\mathbf{Q}^\text{v}$, to compute a degree-$(D + \Delta D)$ triangular factor, $\mathbf{r}$, of an equivalent triangular system, $\mathbf{r} \Pi$.}
\label{fig:nullspace}
\begin{algorithmic}
\State $A_{o\mathbf{k}}^{mn\mathbf{h}} \Leftarrow (\mathbf{e}_m)^\T \mathbf{a}_\mathbf{j} (\delta_o^n \varsigma_\mathbf{k}^{\mathbf{h}\mathbf{i}} - \varsigma_{o\mathbf{k}}^{n\mathbf{h}\mathbf{i}})$
\State $\mathbf{A} \Leftarrow \mathbf{e}_{mn\mathbf{h}} A_{o\mathbf{k}}^{mn\mathbf{h}} (\mathbf{e}_{o\mathbf{k}})^\T$ \Comment{``zero''-cost reshape}
\State $\mathbf{Z} \Leftarrow \mathrm{null}(\mathbf{A}^\T)$ \Comment{i.e., where $\mathbf{A}^\T \mathbf{Z} \approx \mathbf{0}$}
\State $\mathbf{Z}^{\mathbf{h}j} \Leftarrow \mathbf{e}_m ((\mathbf{e}_{mn\mathbf{h}})^\T \mathbf{Z} \mathbf{e}^\mathbf{j}) \mathbf{e}_n^\T$ \Comment{``zero''-cost reshapes}
\State $y_j \Leftarrow \mathbf{Z}^{\mathbf{0}j} \backslash \mathbf{I}$
\State $\mathbf{Q}_\mathbf{h}^\text{v} \Leftarrow (\mathbf{I} \delta_\mathbf{k}^\mathbf{0}) / ((\mathbf{Z}^{\mathbf{i}j} y_j)^\T \varsigma_\mathbf{k}^{\mathbf{h}\mathbf{i}})$
\If{$\mathrm{options}.\mathrm{MaxIterations} > 0$}
\State $\mathbf{E}_\mathbf{k}^\text{v} \Leftarrow @(y, \mathbf{Q}^\text{v})\; \mathbf{Q}_\mathbf{h}^\text{v} (\mathbf{Z}^{\mathbf{i}j} y_j)^\T \varsigma_\mathbf{k}^{\mathbf{h}\mathbf{i}} - \mathbf{I} \delta_\mathbf{k}^\mathbf{0}$
\State $\mathbf{E}_\text{u} \Leftarrow @(y)\; (\mathbf{Z}^{\mathbf{0}j} y_j)^\T \mathbf{Z}^{\mathbf{0}j} y_j - \mathbf{I}$
\State $\SSE^\text{v} \Leftarrow @(y, \mathbf{Q}^\text{v})\; \tr\, (\mathbf{E}_\mathbf{k}^\text{v}(y, \mathbf{Q}^\text{v}))^\T \mathbf{E}_\mathbf{k}^\text{v}(y, \mathbf{Q}^\text{v})$
\State $\SSE_\text{u} \Leftarrow @(y)\; \tr\, (\mathbf{E}_\text{u}(y))^\T \mathbf{E}_\text{u}(y)$
\State $f \Leftarrow @(y, \mathbf{Q}^\text{v})\; \SSE^\text{v}(y, \mathbf{Q}^\text{v}) + \SSE_\text{u}(y)$
\State $G_\text{u}^j \Leftarrow @(y, \mathbf{Q}^\text{v})\; \partial f(y, \mathbf{Q}^\text{v}) / \partial y_j$
\State $G_{mn}^{\text{v}\mathbf{h}} \Leftarrow @(y, \mathbf{Q}^\text{v})\; \partial f(y, \mathbf{Q}^\text{v}) / \partial Q_\mathbf{h}^{\text{v}mn}$
\State $[y, \mathbf{Q}^\text{v}] \Leftarrow \mathrm{fminunc}(f, G_\text{u}, G^\text{v}, y, \mathbf{Q}^\text{v}, \mathrm{options})$
\EndIf
\State $\mathbf{Q}_\text{u}^\mathbf{i} \Leftarrow \mathbf{Z}^{\mathbf{i}j} y_j$
\State $\mathbf{r}_\mathbf{k} \Leftarrow (\mathbf{Q}_\text{u}^\mathbf{i})^* \mathbf{a}_\mathbf{j} \varsigma_\mathbf{k}^{\mathbf{i}\mathbf{j}}$
\end{algorithmic}
\end{algorithm}

Setting aside the unitary factors, the triangular factor, $\mathbf{r}$, if known specifies the triangular system. Using basis vectors, $\mathbf{e}$ and $\mathbf{e}^\T$, to convert between row and index values, one may express the triangular constraints in a homogeneous form:
\begin{align}
\mathbf{r}_\mathbf{k} - \mathbf{e}^\ell(\mathbf{e}_n^\T \mathbf{r}_\mathbf{j} \varsigma_{\ell\mathbf{k}}^{n\mathbf{j}}) &= \mathbf{0}_\mathbf{k} \text{,}
\end{align}
where the variant sigma symbol, $\varsigma$, applies in part to converted row index values. The unknown, $\mathbf{r}$, factors out after introducing a Dirac delta symbol, $\delta$, for index vector substitution:
\begin{align}
(\mathbf{I} \delta_\mathbf{k}^\mathbf{j} - \mathbf{e}^\ell \mathbf{e}_n^\T \varsigma_{\ell\mathbf{k}}^{n\mathbf{j}}) \mathbf{r}_\mathbf{j} &= \mathbf{0}_\mathbf{k} \text{.}
\end{align}
Replacing the unknown vector, $\mathbf{r}$, by its definition brings the known or given vector, $\mathbf{a}$, into the homogeneous equation:
\begin{align}
(\mathbf{I} \delta_\mathbf{k}^\mathbf{j} - \mathbf{e}^\ell \mathbf{e}_n^\T \varsigma_{\ell\mathbf{k}}^{n\mathbf{j}}) (\mathbf{Q}_\text{u}^\mathbf{h})^* \mathbf{a}_\mathbf{i} \varsigma_\mathbf{j}^{\mathbf{h}\mathbf{i}} &= \mathbf{0}_\mathbf{k} \text{,}
\end{align}
where the forward unitary factor, $\mathbf{Q}_\text{u}$, is the only unknown given the incremental parameter, $\Delta D$, specifying its degree.

All matrices and vectors in the homogeneous equation have equivalent scalars of higher degree, where row and column index values convert to tensor index values consistently:
\begin{align}
(\delta_o^p \delta_\mathbf{k}^\mathbf{j} - \delta_o^\ell \delta_n^p \varsigma_{\ell\mathbf{k}}^{n\mathbf{j}}) (Q_\text{u}^{mp\mathbf{h}})^* a_\mathbf{i}^m \varsigma_\mathbf{j}^{\mathbf{h}\mathbf{i}} &= 0_{o\mathbf{k}} \text{,}
\end{align}
One may rearrange the terms in a scalar product expression arbitrarily, in this case to factor out the only unknown, $Q_\text{u}$:
\begin{align}
(Q_\text{u}^{mp\mathbf{h}})^* \delta_n^p a_\mathbf{i}^m (\delta_o^n \delta_\mathbf{k}^\mathbf{j} - \delta_o^\ell \varsigma_{\ell\mathbf{k}}^{n\mathbf{j}}) \varsigma_\mathbf{j}^{\mathbf{h}\mathbf{i}} &= 0_{o\mathbf{k}} \text{,}
\end{align}
Thus, one may summarize the scalar equivalent of the homogeneous equation using a scalar linear function, $A$, of the given vector, $\mathbf{a}$, reshaped into a higher-degree scalar, $a$:
\begin{align}
(Q_\text{u}^{mn\mathbf{h}})^* A_{o\mathbf{k}}^{mn\mathbf{h}} &= 0_{o\mathbf{k}} \text{,} \\
a_\mathbf{i}^m (\delta_o^m \varsigma_\mathbf{k}^{\mathbf{h}\mathbf{i}} - \varsigma_{o\mathbf{k}}^{n\mathbf{h}\mathbf{i}}) &= A_{o\mathbf{k}}^{mn\mathbf{h}} \text{.}
\end{align}

Consequently, the forward unitary factor, $Q_\text{u}$, lies in a null space. It equals a weighted sum of a null-space basis, $Z$, with unknown weights, $y$, both of which require computation:
\begin{align}
Q_\text{u}^{mn\mathbf{h}} &= Z^{mn\mathbf{h}j} y_j \text{.}
\end{align}
The basis depends on the known scalar, $A$. Conjugation of the null-space relationship yields an unconjugated basis, $Z$:
\begin{align}
(A_{o\mathbf{k}}^{mn\mathbf{h}})^* Z^{mn\mathbf{h}j} &= 0_{o\mathbf{k}}^j \text{.}
\end{align}
Algorithm~\ref{fig:nullspace} employs a ``zero''-cost reshape and a conjugate transpose to convert the high-degree scalar, $A$, into a degree-zero matrix, $\mathbf{A}^*$, from which a linear algebra operation yields a null-space basis, $\mathbf{Z}$, that undergoes a further reshape, considering the desired final shape of the unitary factor, $\mathbf{Q}_\text{u}$:
\begin{align}
\mathbf{Q}_\text{u}^\mathbf{i} &= \mathbf{Z}^{\mathbf{i}j} y_j \text{.}
\end{align}

The first unitary identity requires a nonsingular degree-zero end slice, $\mathbf{Q}_\text{u}^\mathbf{0}$, of the degree-$\Delta D$ factor, $\mathbf{Q}_\text{u}$. Assuming an end slice, $\mathbf{Z}^\mathbf{0}$, of the null-space basis, $\mathbf{Z}$, has sufficient rank, one may construct and approximately solve a linear system to initialize the basis weights, $y$, a degree-one scalar here:
\begin{align}
\mathbf{I} &= \mathbf{Z}^{\mathbf{0}j} y_j \text{.}
\end{align}
The \ac{RT} framework, algebra and codesigned software, allows linear system solving via right- or left-division expressions:
\begin{align}
y_j &= \mathbf{Z}^{\mathbf{0}j} \backslash \mathbf{I} \text{.}
\end{align}
Along with initialized weights, the second part of Algorithm~\ref{fig:nullspace} requires an initialized backward factor, $\mathbf{Q}^\text{v}$. The latter follows from a division derived from the second unitary identity:
\begin{align}
\mathbf{Q}_\mathbf{h}^\text{v} &= (\mathbf{I} \delta_\mathbf{k}^\mathbf{0}) / ((\mathbf{Z}^{\mathbf{i}j} y_j)^* \varsigma_\mathbf{k}^{\mathbf{h}\mathbf{i}}) \text{.}
\end{align}

The null-space model, $\mathbf{Z} y$, of the forward unitary factor, $\mathbf{Q}_\text{u}$, ensures the transformed system, $\mathbf{Q}_\text{u}^* \mathbf{a} \varsigma$, satisfies triangularity constraints. Apart from the basis weights, $y$, a complete \ac{Qr} factorization requires the backward unitary factor, $\mathbf{Q}^\text{v}$, where unitary factors zero forward and backward identity errors:
\begin{align}
\mathbf{E}_\text{u} &= (\mathbf{Q}_\text{u}^\mathbf{0})^* \mathbf{Q}_\text{u}^\mathbf{0} - \mathbf{I} \text{,} \\
\mathbf{E}_\mathbf{k}^\text{v} &= \mathbf{Q}_\mathbf{h}^\text{v} (\mathbf{Q}_\text{u}^\mathbf{i})^* \varsigma_\mathbf{k}^{\mathbf{h}\mathbf{i}} - \mathbf{I}_\mathbf{k}^\mathbf{0} \text{.}
\end{align}

One approach, simple to develop, defines and minimizes to zero, if possible, a nonnegative and real objective function, $f$, equal to the \ac{SSE} over forward and backward identities:
\begin{align}
f &= \underbrace{\tr\, (\mathbf{E}_\text{u})^* \mathbf{E}_\text{u}}_{\SSE_\text{u}} + \underbrace{\tr\, (\mathbf{E}_\mathbf{k}^\text{v})^* \mathbf{E}_\mathbf{k}^\text{v}}_{\SSE^\text{v}} \text{.}
\end{align}
For practical reasons, doing so requires at least first derivatives, i.e., gradients, with respect to the variables, $y$ and $\mathbf{Q}^\text{v}$. With complex-valued coefficients, $\mathbf{a}$, this approach would require separate treatment of real and imaginary parts. In this initial treatment, the optimization handles only real coefficients.

As pseudocode, Algorithm~\ref{fig:nullspace} mixes \ac{RT} algebra expressions with MATLAB syntax. Its second part defines function handles for forward and backward errors, $\mathbf{E}_\text{u}$ and $\mathbf{E}^\text{v}$, their \ac{SSE} counterparts, $\SSE_\text{u}$ and $\SSE^\text{v}$, the scalar objective, $f$, and its gradients, $G_\text{u}$ and $G^\text{v}$, which follow using the algebra:
\begin{align}
\underbrace{\frac{\partial f}{\partial y_j}}_{G_\text{u}^j} &= \underbrace{2 \tr\, (\mathbf{E}_\text{u})^\T \frac{\partial\mathbf{E}_\text{u}}{\partial y_j}}_{G_\text{uu}^j} + \underbrace{2 \tr\, (\mathbf{E}_\mathbf{k}^\text{v})^\T \frac{\partial\mathbf{E}_\mathbf{k}^\text{v}}{\partial y_j}}_{G_\text{vu}^j} \text{,} \\
\underbrace{\frac{\partial f}{\partial Q_\mathbf{h}^{\text{v}mn}}}_{G_{mn}^{\text{v}\mathbf{h}}} &= \underbrace{2 \tr\, (\mathbf{E}_\text{u})^\T \mathbf{0}_{mn}^\mathbf{h}}_{G_{mn}^{\text{uv}\mathbf{h}}} + \underbrace{2 \tr\, (\mathbf{E}_\mathbf{k}^\text{v})^\T \frac{\partial\mathbf{E}_\mathbf{k}^\text{v}}{\partial Q_\mathbf{h}^{\text{v}mn}}}_{G_{mn}^{\text{vv}\mathbf{h}}} \text{.}
\end{align}

Applying the differentiation chain rule to obtain the gradient, $G_\text{u}$, of the objective function, $f$, with respect to the forward weights, $y$, yields two terms, $G^\text{uu}$ and $G^\text{vu}$:
\begin{align}
G_j^\text{uu} &= 4 \tr\, (\mathbf{E}^\text{u})^\T (\mathbf{Z}^{\mathbf{0}j})^\T \mathbf{Q}_\text{u}^\mathbf{0} \text{,} \\
G_j^\text{vu} &= 2 \tr\, (\mathbf{E}_\mathbf{k}^\text{v})^\T \mathbf{Q}_\mathbf{h}^\text{v} (\mathbf{Z}^{\mathbf{i}j})^\T \varsigma_\mathbf{k}^{\mathbf{h}\mathbf{i}} \text{.}
\end{align}
Applying the chain rule to obtain the gradient, $G^\text{v}$, with respect to the backward factor, $\mathbf{Q}^\text{v}$, expressed as a higher-degree scalar, $Q^\text{v}$, yields one nonzero term, $G^\text{vv}$, in this case:
\begin{align}
G_{mn}^{\text{vv}\mathbf{h}} &= 2 \tr\, (\mathbf{E}_\mathbf{k}^\text{v})^\T \mathbf{e}_m \mathbf{e}_n^\T (\mathbf{Q}_\text{u}^\mathbf{i})^\T \varsigma_\mathbf{k}^{\mathbf{h}\mathbf{i}} \text{.}
\end{align}
Whether an index, like $j$, exists in the true or false-variant (subscript or superscript) position matters. Index variants in divisors, differentials included, change to complements.

Because the \ac{RT} algebra expresses not only scalar tensors but also matrices and vectors with indices, called matrices and vectors of nonzero degree as opposed to matrix and vector tensors, the \ac{RT} algebra incorporates linear algebra identities like commutation in the trace of a matrix product:
\begin{align}
\tr\, \mathbf{A} \mathbf{B} &= \tr\, \mathbf{B} \mathbf{A} \text{.}
\end{align}
Skipping intermediate steps that involve other \ac{RT} algebra identities, further simplifications are possible before implementation in accompanying software, the RTToolbox:
\begin{align}
G_j^\text{uu} &= 4 \tr\, \mathbf{Q}_\text{u}^\mathbf{0} (\mathbf{E}^\text{u})^\T (\mathbf{Z}^{\mathbf{0}j})^\T \text{,} \\
G_j^\text{vu} &= 2 \tr\, \varsigma_\mathbf{k}^{\mathbf{h}\mathbf{i}} (\mathbf{E}_\mathbf{k}^\text{v})^\T \mathbf{Q}_\mathbf{h}^\text{v} (\mathbf{Z}^{\mathbf{i}j})^\T \text{,} \\
G_{mn}^{\text{vv}\mathbf{h}} &= \mathbf{e}_n^\T (\mathbf{G}_\mathbf{h}^\text{vv})^\T \mathbf{e}_m \text{,} \\
\mathbf{G}_\mathbf{h}^\text{vv} &= 2 \cdot \mathbf{E}_\mathbf{k}^\text{v} (\varsigma_\mathbf{k}^{\mathbf{h}\mathbf{i}})^\T \mathbf{Q}_\text{u}^\mathbf{i} \text{.}
\end{align}
These simplifications consider operand dimension sizes and left-to-right binary multiplication in MATLAB, an object-oriented programming language. The RTToolbox also replaces the trace-after-multiply with \ac{BSX}-and-sum operations.

Parameters for \texttt{fminunc} of MATLAB's Optimization Toolbox, employed in Algorithm~\ref{fig:nullspace}, include objective function and step size tolerances, along with a maximum number of iterations. An end user may accept default values or specify them as options. When \texttt{fminunc} terminates, the \ac{Qr} factorization succeeds if the objective function, $f$, not only has stopped minimizing but also equals zero to within its tolerance.

\subsection{Index Compression}

Even assuming sparse representation and the summation of equivalent monomials, the coefficient tensor, $\mathbf{a}$, that defines in canonical form a polynomial system, $\mathbf{a} \Pi$, may be inefficient in terms of the memory required per nonzero value to represent associated index values. The index range of (or the number of elements in) a degree $D$ scalar, $\mathbf{a}$, over $N$ variables equals $(N + 1)^D$, which necessitates at least $\lceil D \log_2(N + 1) \rceil$ bits to represent the index values internally as linearized index values. Due to incremental degree, the triangular system, $\mathbf{r} \Pi$, involves a coefficient tensor, $\mathbf{r}$, of even higher degree, $D + \Delta D$. If the incremental degree were to grow in proportion to $e^N$ then so would the number of bits per linearized index value.

Equations that defined polynomial systems and unitary transforms remain unchanged if coefficients get replaced with the same after summation of equivalent monomials. Therefore, scalar indices, $h$, $i$, $j$, and $k$, that each range over corresponding unique monomials may replace the index vectors, $\mathbf{h}$, $\mathbf{i}$, $\mathbf{j}$, and $\mathbf{k}$. Additional bookeeping is required to compute expressions, using such compressed indices, that also involve a variant sigma symbol. Because the number of unique monomials of at most degree $D$ over $N$ variables is $(N + D)$-choose-$D$ or $(D + N)$-choose-$N$, one can prove that the minimum number of bits per nonzero to represent a linearized index value reduces exponentially with respect to the degree:
\begin{align}
\bigg\lceil \log_2\binom{D + N}{N} \bigg\rceil &= \bigg\lceil \log_2 \frac{(D + N) \cdots (D + 1)}{N (N - 1) \cdots 1} \bigg\rceil \text{,} \\
&\le \big\lceil N \log_2(D + N) \big\rceil \text{.}
\end{align}

A tighter upper bound, $\lceil N \log_2(D + 1) \rceil$, can replace the loose one, $\lceil N \log_2(D + N) \rceil$. In a degree-$D$ monomial, setting aside without prejudice a constraint on the sum of the powers, each of $N$ variables goes to a power of $0$ to at most $D$. For the special case of a degree-two polynomial over $N$ variables, Figure~\ref{fig:indexspace} illustrates the difference in representation. Here, one can more easily count the unique monomials as $(N + 2)$-choose-$2$ rather than $(2 + N)$-choose-$N$, although equal.

\begin{figure}[t]
\centering
\begin{tabular}{c}
$\bracket{x_1 \\ x_2 \\ \vdots \\ x_N \\ 1}^\T \underbrace{\bracket{a_{11} & a_{12} & \cdots & a_{1N} & a_{10} \\ 0 & a_{22} & \cdots & a_{2N} & a_{20} \\ \vdots & \vdots & \ddots & \vdots & \vdots \\ 0 & 0 & \cdots & a_{NN} & a_{N0} \\ 0 & 0 & \cdots & 0 & a_{00}}}_{\text{(a) index range } \propto\, (N + 1)^D} \bracket{x_1 \\ x_2 \\ \vdots \\ x_N \\ 1} = 0$ \\
$\underbrace{\bracket{a_{11} & \cdots & a_{10} & a_{22} & \cdots & a_{00}}}_{\text{(b) index range } \propto\, (D + 1)^N} \bracket{x_1 x_1 \\ \vdots \\ x_1 \\ x_2 x_2 \\ \vdots \\ 1} = 0$ \\
\end{tabular}
\mycaption{Representations of a multivariate degree-two polynomial.}{(a) The outer-product model naturally represents this quadratic form. A summing of related coefficients increases the sparsity but does not reduce the index range. (b) At the cost of some bookkeeping, the index range reduces by representing nonzeros in a way proportional to the number of unique monomials.}
\label{fig:indexspace}
\end{figure}


\section{Concept Demos}
\label{sec:concept demos}

This section presents two polynomial systems, chosen for educational value, to illustrate the proposed \ac{RT} framework and \ac{Qr} factorization approach to triangularization. One, \texttt{boxdims}, solves without iteration. The other, \texttt{arm2seg}, needs iterations that an executable accelerates. While depicting key \ac{RT} algebra concepts, the section introduces codesigned \ac{RT} software.

\subsection{Box Dimensions}

After installing the RTToolbox, \ac{R2026}, an end user may compile a provided file, \texttt{spmex1.c}, to produce an optional executable, \texttt{spmex1}. Given availability of a C compiler in MATLAB, a simple \texttt{mex} command would produce the executable, possibly without a use case via a \texttt{-D}... option.

\begin{verbatim}
>> % mex -R2018a -Drrlu spmex1.c
>> % mex -R2018a -Dbsx spmex1.c
>> mex -R2018a spmex1.c
\end{verbatim}

Figure~\ref{fig:boxdims} illustrates a box dimensions problem. Three homogeneous polynomial equations depend on required numerical parameters: total length, total surface area, and total volume. One vector, $\mathbf{x}$, represents the three scalar variables of interest: length, width, and height. The RTToolbox includes a function, \texttt{boxdims}, that returns a sparse \texttt{tensor} object, \texttt{a}, which defines all ten nonzero coefficients, $\mathbf{a}$, of the system, $\mathbf{a} \Pi$.

\begin{figure}[t]
\centering
\begin{tabular}{c}
\insertpdf{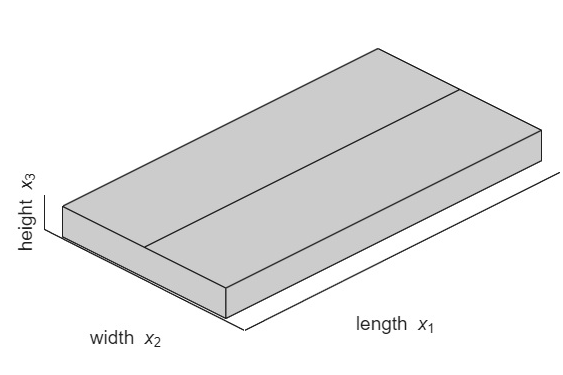} \\
$\underbrace{\bracket{x_1 + x_2 + x_3 - 1 \\ 2 x_1 x_2 + 2 x_1 x_3 + 2 x_2 x_3 - 0.5 \\ x_1 x_2 x_3 - 0.01}}_{\text{(b) original } \mathbf{a}_\mathbf{j} \Pi^\mathbf{j}(\mathbf{x}, 1)} = \underbrace{\bracket{0 \\ 0 \\ 0}}_{\mathbf{0}}$ \vspace{1ex} \\
$\underbrace{\bracket{x_1 + x_2 + x_3 - 1 \\ -2 x_2^2 - 2 x_2 x_3 + 2 x_2 - 2 x_3^2 + 2 x_3 - 0.5 \\ x_3^3 - x_3^2 + 0.25 x_3 - 0.01}}_{\text{(c) triangular } \mathbf{r}_\mathbf{k} \Pi^\mathbf{k}(\mathbf{x}, 1)} = \underbrace{\bracket{0 \\ 0 \\ 0}}_{\mathbf{0}}$ \\
\end{tabular}
\mycaption{The box dimensions problem and algebraic geometry.}{(a) Suppose we know the total length, $1\,\m$, the surface area, $0.5\,\m^2$, and the volume, $0.01\,\m^3$, of a rectangular box. (b) We can make a degree-three polynomial system over the three dimensions of interest, $x_1$, $x_2$, and $x_3$, in $\m$. (c) The proposed \acs{Qr} factorization yields a degree-three triangular system that makes it easier to solve, in sequence, for the unknown height, width, and length.}
\label{fig:boxdims}
\end{figure}

\begin{verbatim}
>> a = boxdims(1,0.5,0.01);
>> disp(a)
   ...
   (2,1,2,3,4)     2
   (2,1,4,4,4)   -0.5000
   (3,1,1,2,3)     1
   (3,1,4,4,4)   -0.0100
\end{verbatim}

With \ac{R2026}, \texttt{tensor} objects have a \texttt{qr} method for \ac{Qr} factorization. When invoked with one, two, three, or four output arguments, it returns the triangular factor, \texttt{r}, the forward unitary and triangular factors, \texttt{Qu} and \texttt{r}, the three factors, \texttt{Qv}, \texttt{Qu}, and \texttt{r}, or the three factors and a backward identity value, \texttt{Qv}, \texttt{Qu}, \texttt{r}, and \texttt{Id}, respectively. Each factor, and the identity, is a \texttt{tensor} having, as with \ac{R2024}, a \texttt{degree} method.

\begin{verbatim}
>> r = qr(a,2,'Display','off');
>> assert(degree(r) == degree(a)+2)
>> disp(r)
   ...
   (3,1,3,3,3,4,4)     1
   (3,1,3,3,4,4,4)    -1
   (3,1,3,4,4,4,4)    0.2500
   (3,1,4,4,4,4,4)   -0.0100
\end{verbatim}

The \texttt{qr} method accepts name-value argument pairs that set parameters of an internal optimization that uses the \texttt{fminunc} function of MATLAB's Optimization Toolbox. All such arguments pass unchanged to \texttt{fminunc} as an \texttt{options} object, constructed via \texttt{optimset}. Using a \texttt{'Display'} option, which defaults to \texttt{'final'} (not \texttt{'off'} or \texttt{'none'}), one can clearly demonstrate that, for the \texttt{boxdims} problem, the proposed \ac{Qr} factorization does not require iteration.

\begin{verbatim}
>> [Qu,r] = qr(a,2,'Display','iter');
...
Initial point is a local minimum.
...
\end{verbatim}

While a \texttt{simplify} method performs without multiplication the equivalent of multiplication with a variant sigma symbol, $\varsigma$, the RTToolbox also has a function, \texttt{vsigma}, to construct a variant sigma symbol as a sparse \texttt{tensor} object, \texttt{vs}. The method and function facilitate \ac{Qr} verification.

\begin{verbatim}
>> [Qv,Qu,r] = qr(a,2,'Display','off');
>> Qr = simplify(Qv*r); % 'mindeg'
>> assert(isequal(a,Qr(index(a))))
>> [vs,~,k] = vsigma(Qu',a);
>> assert(isequal(r(k),Qu'*a*vs))
\end{verbatim}

As triangular coefficients, \texttt{r}, involve monomial collection after \texttt{tensor} multiplication, the highest monomial degree associated with a nonzero coefficient may be less than the number of indices, \texttt{degree(r)}. With its default \texttt{'mindeg'} option, as opposed to \texttt{'argdeg'}, the \texttt{simplify} method returns an equivalent \texttt{tensor} of minimum \texttt{degree}.

\begin{verbatim}
>> r = simplify(r);
>> disp(r)
   ...
   (3,1,3,3,3)     1
   (3,1,3,3,4)    -1
   (3,1,3,4,4)    0.2500
   (3,1,4,4,4)   -0.0100
\end{verbatim}

The \texttt{tensor} class has a \texttt{roots} method that for a univariate polynomial invokes MATLAB's \texttt{roots} function, passing along the coefficients as a sparse row-shaped vector. It returns the result as a column-shaped vector. Because of symmetries in the box dimensions problem, the three solution values, \texttt{x3}, of the end triangular equation make a solution vector, \texttt{x}.

\begin{verbatim}
>> k = index(r);
>> x3 = roots(r(end,1,k));
>> x = sort(x3,'descend')
x =
    0.6264
    0.3244
    0.0492
\end{verbatim}

Another method, \texttt{polyval}, evaluates a polynomial system from its coefficients, specified as a \texttt{tensor} object. Values of variables must be specified as a non-\texttt{tensor} matrix or column vector. Matrix variables interpret columnwise. For each input column, the \texttt{polyval} method evaluates the system, returning an output column. Output rows correspond to system equations. Input columns should be augmented with ones or an additional argument of ones or one should be given.

\begin{verbatim}
>> polyval(a,x,1)
ans =
   1.0e-15 *
    0.2220
    0.2220
    0.0035
\end{verbatim}

Requiring MATLAB's Symbolic Math Toolbox for this use, \texttt{polyval} accepts \texttt{sym} variables, in which case it can return a polynomial system, upon transposition, in the form expected by Symbolic Math Toolbox functions like \texttt{gbasis}. For the problem at hand, the \ac{Qr} factorization computes a reduced \ac{GB} equivalent, differences amounting only to scale factors and insignificant round-off errors. Figure~\ref{fig:boxdims} presents the triangular system obtained via \ac{Qr} factorization with the RTToolbox.

\begin{verbatim}
>> symx = sym('x',[3 1]);
>> symr = polyval(r,symx,1);
>> syma = polyval(a,symx,1);
>> simplify(symr.'*diag([1 -2 100])- ...
    gbasis(syma.','MonomialOrder','lex'))
ans =
[0, 0, 0]
\end{verbatim}

Because forward and backward unitary factors, \texttt{Qu} and \texttt{Qv}, satisfy the unitary identities, their end slices, \texttt{Qu(:,:,end)} and \texttt{Qv(:,:,end)}, equate and have the form of an ordinary unitary matrix. With the \texttt{polyval} method and the Symbolic Math Toolbox, one can compute the forward and backward transforms, \texttt{U} and \texttt{V}, implied by these factors. They incorporate variables. Due to the backward unitary identity, the transform product, \texttt{V*U}, equals an ordinary identity matrix, \texttt{I}.

\begin{verbatim}
>> U = polyval(Qu',symx,1); % symU
>> V = polyval(Qv,symx,1); % symV
>> I = eye(3,3,'sym'); % symI
>> assert(isequal(simplify(V*U),I))
\end{verbatim}

\subsection{Two-Segment Arm}

Figure~\ref{fig:arm2seg} illustrates a two-segment robotic arm problem. Instead of formulating it in terms of the actuator angles of interest, $\phi_1$ and $\phi_2$, formulating it in terms of displacements, $x_1$ to $y_2$, produces a polynomial system, $\mathbf{a} \Pi$. From a solution vector, $[y_2, x_2, y_1, x_1]^\T$, in this case giving primaacy to the first actuator, four-quadrant inverse tangents yield angles.

\begin{figure}[t]
\centering
\begin{tabular}{c}
\insertpdf{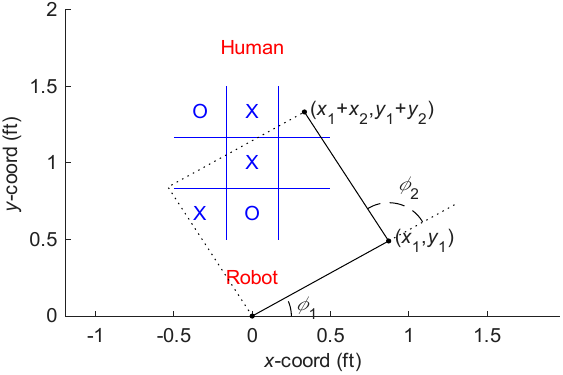} \vspace{1ex} \\
$\underbrace{\bracket{x_1 + x_2 - 1 / 3 \\ y_1 + y_2 - 4 / 3 \\ x_1^2 + y_1^2 - 1 \\ x_2^2 + y_2^2 - 1}}_{\text{(b) original } \mathbf{a}_\mathbf{j} \Pi^\mathbf{j}(\mathbf{w}, 1)} = \underbrace{\bracket{0 \\ 0 \\ 0 \\ 0}}_{\mathbf{0}}$ \vspace{1ex} \\
$\underbrace{\mathbf{R} \bracket{8 / 3 \,y_2 - 2 / 3 \,x_1 - 5 / 3 \\ x_2 + x_1 - 1 / 3 \\ 8 / 3 \,y_1 + 2 / 3 \,x_1 - 17 / 9 \\ 68 / 9 \,x_1^2 - 68 / 27 x_1 - 287 / 81}}_{\text{(c) triangular } \mathbf{r}_\mathbf{k} \Pi^\mathbf{k}(\mathbf{w}, 1)} = \underbrace{\bracket{0 \\ 0 \\ 0 \\ 0}}_{\mathbf{0}}$ \\
\end{tabular}
\mycaption{The two-segment arm problem and algebraic geometry.}{(a) A robot playing a game has an arm with two $1\,\ft$ segments and one end at the origin. (b) To direct the other end to a point in a plane, $(1 / 3, 4 / 3)\,\ft$, yields a degree-two polynomial system over four variables, $x_1$, $y_1$, $x_2$, and $y_2$, in $\ft$. (c) A \acs{Qr} factorization exists to produce a degree-two triangular system for calculating, in sequence, the primary and secondary actuator angles, $\phi_1$ and $\phi_2$.}
\label{fig:arm2seg}
\end{figure}

Here, the \ac{Qr} factorization does not fully converge within the default maximum number of iterations, 400, of the \texttt{fminunc} function. The \texttt{arm2seg} function of the RTToolbox creates a required sparse \texttt{tensor} object. To assess convergence, one can examine the residual errors of the unitary identities, something the \texttt{qr} method with four output arguments facilitates. As before, this work does not address how to predict the minimum incremental degree parameter, $\Delta D$, two in both demos.

\begin{verbatim}
>> a = arm2seg(1/3,4/3);
>> [Qv,Qu,r,Id] = qr(a,2);
Solver stopped prematurely.
...
>> iseq = @(a,b,n) isequal( ...
    round(a,n),round(b,n));
>> I = speye(size(Qu,2));
>> Qu0 = Qu(:,:,end);
>> assert(iseq(Qu0'*Qu0,I,3))
>> [vs,~,k] = vsigma(Qv,Qu');
>> assert(iseq(Qv*Qu'*vs,Id(k),2))
\end{verbatim}

To increase the accuracy of the unitary factor identities from three and two digits after the decimal point to six and four digits, one can increase the \texttt{'MaxIterations'} option, e.g., by setting it to \texttt{Inf}. After 1226 iterations, the \texttt{fminunc} function terminates on account of a gradient size tolerance, which can also be changed. Unless otherwise stated, additional results in this section refer to default \texttt{fminunc} options.

With this release, the \texttt{qr} method does not automatically calculate a tolerance for significant nonzeros nor zero insignificant ones. As most coefficients are on the \texttt{1e0} order, one can eliminate those below the \texttt{1e-15} order. After doing so, the number of nonzeros in the triangular system reduces from 174 or more (exact in MATLAB Online) to 118. The roots of the end univariate equation prove accurate, given a known solution via \iac{GB} method, to 15 digits after the decimal point.

\begin{verbatim}
>> assert(nnz(r) >= 174)
>> r(abs(r) < 1e-15) = 0;
>> assert(nnz(r) == 118)
>> x1Qr = roots(r(end,1,index(r)));
>> x1GB = roots([68/9 -68/27 -287/81]);
>> assert(iseq(x1Qr,x1GB,15))
\end{verbatim}

Because there are more significant issues to discuss at this time regarding \ac{Qr} factorization, this work does not investigate back substitution into the triangular system to solve for other variables. Admittedly, the weakest part of Algorithm~\ref{fig:nullspace} is the second part involving the \texttt{fminunc} optimization. The motivation for that choice is that it was a simple way to complete the algorithm that would not take a lot of space to explain. One can easily verify, with the included software, that the end univariate equation due to the first and third parts of the algorithm alone, i.e., with zero iterations, is correct to 15 digits after the decimal point. Figure~\ref{fig:arm2seg} illustrates (solid line) one solution, 0.8715, to $x_1$ (and also to $y_1$, $x_2$, and $y_2$).

\begin{verbatim}
>> r = qr(a,2,'MaxIterations',0);
>> assert(nnz(r) >= 48)
>> r(abs(r) < 1e-15) = 0;
>> assert(nnz(r) == 19)
>> x1Qr = roots(r(end,1,index(r)));
>> assert(iseq(x1Qr,x1GB,15))
\end{verbatim}

With zero iterations, the \texttt{qr} method for the \texttt{arm2seg} system yields unitary factors that satisfy the forward identity to 15 digits after the decimal point but do not satisfy the backward one. The \texttt{arm2seg} function can provide handcrafted unitary factors that satisfy the backward identity to 15 digits. A \texttt{qq} method of the \texttt{tensor} class can transform unitary factors that accurately satisfy the backward identity into ones that accurately satisfy both identities. This non-iterative method uses two 180\textdegree{} matrix rotations (similar to but not transpositions), Cholesky factorization of a guaranteed positive-definite matrix, matrix multiplication, and matrix division. In this manner, one can constructively prove that a sparser \ac{Qr} triangularization, having 18 nonzeros, exists here that equates to a reduced \ac{GB} triangularization via an upper-triangular matrix, \texttt{R} (denoted $\mathbf{R}$ in Figure~\ref{fig:arm2seg}), that the \texttt{qq} method can provide. At present, the released \ac{Qr} factorization algorithm does not take advantage of these noteworthy \ac{Qr} properties identified by this work.

\begin{verbatim}
>> [a,Qu,Qv,Id] = arm2seg(1/3,4/3);
>> [Qu,Qv,R] = qq(Qu,Qv);
>> r = Qu'*a*vsigma(Qu',a);
>> assert(nnz(r) >= 31)
>> r(abs(r) < 1e-16) = 0;
>> assert(nnz(r) == 18)
...
\end{verbatim}

Table~\ref{tab:tictoctime} summarizes the time required by the \texttt{qr} method to compute a factorization and triangularization for the \texttt{boxdims} and \texttt{arm2seg} systems. Averaged over ten runs, times were measured with \texttt{tic} and \texttt{toc} in MATLAB Online. Connected to the \texttt{fminunc} optimization, times primarily depend on the actual number of iterations, which for \texttt{boxdims} is zero regardless of \texttt{options.MaxIterations}. Times secondarily depend on unavailability, \texttt{\% mex -R2018a spmex1}, of the \texttt{spmex1} executable or either of its two main use cases.

\begin{table}[t]
\centering
\mycaption{Elapsed time vs.\ number of iterations.}{Unlike the box dimensions problem, the two-segment arm problem requires multiple iterations to compute a \acs{Qr} factorization according to the proposed algorithm. Time required for initial null-space calculation and subsequent optimization was lowered by a developed \acs{MEX} function, \texttt{spmex1}.}
\label{tab:tictoctime}
\vspace{1ex}
\begin{tabular}{lcccc}
\hline
\acs{Qr} concept demo: & \multicolumn{2}{c}{\texttt{boxdims}} & \multicolumn{2}{c}{\texttt{arm2seg}} \\
\acs{GB} method (\texttt{gbasis}): & \multicolumn{2}{c}{0.01 (s)} & \multicolumn{2}{c}{0.01 (s)} \\
\hline
\texttt{\% mex -R2018a spmex1} & 3.93 & 4.01 & 9.14 & 29.7 \\
\texttt{mex -R2018a -Drrlu spmex1} & 3.85 & 3.92 & 9.13 & 21.6 \\
\texttt{mex -R2018a -Dbsx spmex1} & 0.03 & 0.05 & 0.03 & 19.7 \\
\texttt{mex -R2018a spmex1} & 0.03 & 0.04 & 0.03 & 15.1 \\
\hline
\texttt{options.MaxIterations} & 0 & 400 & 0 & 400 \\
\hline
\end{tabular}
\end{table}

Without \texttt{spmex1} or its \texttt{'rrlu'} use case, the \texttt{qr} method uses MATLAB's \texttt{rref} funtion to compute a precursor to a null space, from which it computes the null space. The sparse use case, \texttt{'rational'}, of MATLAB's \texttt{null} function works the same way. The \texttt{qr} method allows \texttt{null} function use. With it, times (not shown) are essentially identical to the worst two lines of times in Table~\ref{tab:tictoctime}, \texttt{\% mex} ... \texttt{spmex1} and \texttt{mex} ... \texttt{-Drrlu spmex1}. The second-worst line benefits from the \texttt{'bsx'} use case of \texttt{spmex1}, which accelerates expressions in gradient calculations of the \texttt{fminunc} optimization.

\begin{verbatim}
>> r = qr(a,2,@(A) null(A,'rational'));
\end{verbatim}

Table~\ref{tab:tictoctime} shows that, for the two concept demos, the \texttt{qr} method takes slightly and much longer than the \texttt{gbasis} function. In one demo, \texttt{boxdims}, the triangularizations are equivalent. In the other case, the initial guess is close enough to \iac{GB} solution in terms of the end univariate equation but relatively far, to \iac{GB} equivalent solution that exists, for a quasi-Newton optimization. The \texttt{gbasis} function also requires no incremental degree parameter. By examining the \texttt{gbasis} code to the extent possible, most of its operation happens in a \texttt{mupad} executable. In contrast, a significant part of the \texttt{qr} method uses interpreted MATLAB statements, with acclerations to null-space and gradient calculations via the developed \texttt{spmex1} executable. The sparse null-space calculation, a linear algebra operation, outperforms MATLAB's own functions, \texttt{null} or rather \texttt{rref}, for the same purpose.


\section{Tensor Software}
\label{sec:tensor software}

After contrasting the RTToolbox, \ac{R2026}, to a prior release, \ac{R2024}, this section elaborates on how the toolbox facilitates investigation, in MATLAB, of the proposed \ac{Qr} factorization approach to numerical algebraic geometry. A \texttt{sparse1} class and an \texttt{spmex1} executable, two of multiple software developments, help to express and accelerate \emph{sparse} tensor algorithms that leverage a dual-variant index notation, the \ac{RT} algebra.

\subsection{Toolbox Release}

Table~\ref{tab:rttoolbox} presents the initial release, \ac{R2024}, and the current release, \ac{R2026}, of the RTToolbox. With MATLAB and the RTToolbox, \texttt{tensor} and \texttt{index} objects may be constructed and used in innumerable ways consistent with the \ac{RT} algebra, now extended to algebraic geometry via \ac{Qr} factorization.

\begin{table}[t]
\centering
\mycaption{Composition of the open-source RTToolbox, \acs{R2026}.}{Compared to the prior release, \acs{R2024} \cite{Joseph2024a}, this release updates the \texttt{tensor} class and introduces a \texttt{sparse1} class and an \texttt{spmex1} executable. With an \texttt{index} class and some functions, one can investigate tensor algorithms in MATLAB, like the proposed \acs{Qr} factorization, using a dual-variant index notation.}
\label{tab:rttoolbox}
\vspace{1ex}
\begin{tabular}{lll}
\hline
M- or C-file & Summary & Note \\
\hline
\texttt{tensor.m} & Definition of \texttt{tensor} objects & Revised \\
\texttt{tensorTest.m} & Tests of \texttt{tensor} objects & (\acs{R2024}) \\
\texttt{index.m} & Definition of \texttt{index} objects & (\acs{R2024}) \\
\texttt{indexTest.m} & Tests of \texttt{index} objects & (\acs{R2024}) \\
\texttt{page*.m} & Pagewise helper functions & (\acs{R2024}) \\
\hline
\texttt{sparse1.m} & Definition of \texttt{sparse1} objects & Added \\
\texttt{sparse1Test.m} & Tests of \texttt{sparse1} objects & Added \\
\texttt{spmex1.c} & Accelerate \texttt{sparse1} methods & Added \\
\texttt{spmex1Test.m} & Tests of \texttt{spmex1} executable & Added \\
\hline
\texttt{kdelta.m} & Kronecker delta ($\delta$) sparse tensor & Added \\
\texttt{vsigma.m} & Variant sigma ($\varsigma$) sparse tensor & Added \\
\hline
\end{tabular}
\end{table}

Each \texttt{tensor} object has two properties, \texttt{indices} and \texttt{entries}. The latter, \texttt{entries}, for all \ac{R2024} intents and purposes had to be a dense array of class \texttt{double}, MATLAB's primary data type or built-in class. Dimension sizes of this array represented, in sequence, row-and-column dimension sizes followed by additional dimension sizes associated with an \texttt{index} vector, the protected \texttt{indices} property.

This work applies and extends the \ac{RT} algebra to model canonical and triangular polynomial systems, as well as relationships between the two via generalized unitary transforms. It also leverages the \ac{RT} algebra to develop and explain an initial \ac{Qr} factorization algorithm to triangularize a canonical polynomial system. With the RTToolbox, \ac{R2026}, one can investigate these models and the expressions and statements that make up the algorithm with programmatic and computational efficiency in MATLAB, facilitating further development.

The requirement for programmatic efficiency was already met by the RTToolbox, \ac{R2024}. That release's \texttt{tensor} and \texttt{index} classes beautifully expressed the dual-variant index notation in the thereby extended MATLAB programming language. Expressions computed with relatively small abstraction penalties, identified primarily with M-file or interpreted set-theoretic functions of MATLAB, used to parse \texttt{index} scalars and vectors, because numerical computations were performed using mainly built-in or compiled functions of MATLAB.

Although MATLAB's primary data type, class \texttt{double}, supports dense and sparse arrays of dimension sizes $M \times N$, there is no MATLAB sparse \ac{MDA} type. The prior RTToolbox release could only construct degree-zero sparse \texttt{tensor} objects, i.e., having \texttt{entries} of class \texttt{double} with dimension sizes $M \times N$. Operations that required intermediate or final nonzero-degree sparse tensors would fail to compute.

For computational efficiency in sparse tensor representation and manipulation, consistent with the \ac{RT} algebra, this work developed a sparse \ac{MDA} class, \texttt{sparse1}, and an executable function, \texttt{spmex1}. They accompany revisions to the \texttt{tensor} class not specific to polynomial systems or \ac{Qr} factorization. The new class comprises one file of interpreted MATLAB code. The new function comprises one file of compiled C code. To compile the \ac{MEX} function, easily done by installing the RTToolbox from MATLAB Central File Exchange using the MATLAB \ac{IDE}'s Add Package button, the user must also have configured a C/C++ compiler when MATLAB was installed or must follow prompts to do so by the \ac{IDE}'s Add Package task. The toolbox can operate without the \ac{MEX} function.

The \texttt{sparse1} class provides what looks to other MATLAB code, like the prior class \texttt{tensor}, as a class \texttt{double} object having the \texttt{sparse} attribute, i.e., \texttt{issparse(obj)} equals \texttt{true}, but that allows construction and operation consistent with \acp{MDA}. Polynomial systems and \ac{Qr} factorization aside, almost no changes are required, therefore, to class \texttt{tensor} and no changes are required, whatsoever, to class \texttt{index}.

The changes to class \texttt{tensor} are of three kinds. One is the introduction of familiar sparse \texttt{double} methods, like \texttt{issparse}, that simply apply to the \texttt{entries} property. There was no reason to include such methods in the initial release, \ac{R2024}, of the RTToolbox as it was only introduced with dense numerical examples. The second kind of change involved editing and adding protected \texttt{tensor} methods to perform exactly the same functionality as before but, for the case of sparse \texttt{entries}, in a manner proportional to the number of nonzeros and not the number of elements.

The third kind of changes concern algebraic geometry methods, like \texttt{qr} and \texttt{polyval}, that overload the names of numerical methods in MATLAB for \texttt{double} arrays. When invoked on what are actually \texttt{tensor} objects, these methods implement the \ac{Qr} factorization and compute a polynomial function evaluation using polynomial system coefficients represented according to the proposed tensor algebra.

The \texttt{qr} method always returns sparse \texttt{tensor} objects. Its internal computations always focus on nonzeros regardless of their proportion. The \texttt{polyval} method always returns a degree-zero vector or matrix, meaning it could be represented by class \texttt{double}, dense or sparse. Favouring \ac{RT} algebra, it always returns a \texttt{tensor} object. Strictly speaking, the new \texttt{tensor} methods accept dense \texttt{tensor} arguments and even a mixture of class \texttt{tensor} and \texttt{builtin} class \texttt{double} arguments, dense or sparse. However, for correct operation \ac{MDA} arguments must be supplied as \texttt{tensor} objects, which have \texttt{indices} as a property beyond \texttt{entries}.

Additional toolbox functions, in Table~\ref{tab:rttoolbox}, produce Kronecker delta and varsigma symbols, $\delta$ and $\varsigma$, of arbitrary degree and dimension sizes. They return sparse \texttt{tensor} objects.

\subsection{Sparse Tensors}

Table~\ref{tab:sparse1} lists the public methods, operators included, of the \texttt{sparse1} class, \ac{R2026}. The table groups methods to organize explanation of technical challenges solved by the design.

\begin{table}[t]
\centering
\mycaption{Public methods of an RTToolbox sparse \acs{MDA} class.}{A sparse \texttt{tensor} object of nonzero degree requires a sparse \acs{MDA} class to represent and operate on its \texttt{entries}, alongside its \texttt{indices}. The \texttt{sparse1} class leverages \acs{1D}, and sometimes \acs{2D}, sparse arrays of MATLAB to define sparse \acs{MDA} objects.}
\label{tab:sparse1}
\vspace{1ex}
\begin{tabular}{l}
\hline
Public methods (\texttt{sparse1} class, *\texttt{spmex1} acceleration) \\
\hline
Constructor, \emph{etc.} \\
\hspace{1em} \verb|sparse1| (1--6 args in), \verb|find|, \verb|disp|, \verb|full|, \verb|sparse|, \verb|logical|, \\
\hspace{1em} \verb|double|, \verb|ndims|, \verb|size|, \verb|end|, \verb|numel|, \verb|isempty|, \verb|length|, \verb|nnz|, \\
\hspace{1em} \verb|nzmax|, \verb|issparse|, \verb|isreal|, \verb|class|, \verb|isa|, \verb|subsindex| \\
Unary operations \\
\hspace{1em} \verb|subsref|, \verb|permute|, \verb|reshape|, \verb|sum|, \verb|max|, \verb|min|, \verb|all|, \verb|any|, \verb|abs|, \\
\hspace{1em} \verb|uminus|, \verb|uplus|, \verb|conj|, \verb|sqrt|, \verb|log|, \verb|not|, \verb|real|, \verb|imag|, \verb|round|, \\
\hspace{1em} \verb|pagetranspose|, \verb|pagectranspose|, \verb|pagetrace|, \verb|pagediag|, \\
\hspace{1em} \verb|transpose|, \verb|ctranspose|, \verb|trace|, \verb|diag|, \verb|rref|*, \verb|null|* \\
Binary operations \\
\hspace{1em} \verb|subsasgn|, \verb|plus|*, \verb|minus|*, \verb|eq|*, \verb|ne|*, \verb|lt|*, \verb|gt|*, \verb|le|*, \verb|ge|*, \\
\hspace{1em} \verb|and|, \verb|or|, \verb|times|*, \verb|ldivide|*, \verb|rdivide|*, \verb|power|, \verb|complex|, \\
\hspace{1em} \verb|pagemtimes|, \verb|pagemldivide|, \verb|pagemrdivide|, \\
\hspace{1em} \verb|mtimes|, \verb|mldivide|, \verb|mrdivide| \\
$N$-ary operations \\
\hspace{1em} \verb|isequal|, \verb|cat|, \verb|horzcat|, \verb|vertcat| \\
\hline
\end{tabular}
\end{table}

To construct \ac{1D} or \ac{2D} arrays having \texttt{double} or \texttt{logical} entries, the \texttt{sparse1} class follows the behaviour of MATLAB's \texttt{sparse} function. Because the latter accepts one, two, three, five, and six arguments,  the \texttt{sparse1} constructor accepts one to six arguments. A new four-argument form, \texttt{sparse1(lk,nz,dims,nzm)}, constructs an object of dimension sizes \texttt{dims}, possibly for a sparse \ac{MDA}, with storage for at least \texttt{nzm} nonzero values \texttt{nz}, at linear indices \texttt{lk}. Internally, a \texttt{sparse1} object uses \iac{1D} sparse MATLAB array associated with however many dimension sizes.

Deconstruction and display methods, \texttt{find} and \texttt{disp}, use a MATLAB function, \texttt{ind2sub}, to compute index values, \texttt{i}, \texttt{j}, etc., from linear index values, \texttt{lk}. Methods \texttt{full} and \texttt{sparse} cast the \texttt{sparse1} object to a \texttt{full} and \texttt{sparse} array, where possible, of type \texttt{double} or \texttt{logical}. Methods \texttt{logical} and \texttt{double} keep the \texttt{sparse1} object but cast its nonzero values to type \texttt{logical} and \texttt{double}, without altering the number of nonzeros or their linear index values. Methods \texttt{ndims} to \texttt{isa} return details like the number of dimensions or if nonzero values are of a given class.

The \texttt{subsindex} method casts a \texttt{sparse1} object, \texttt{obj}, to a MATLAB \texttt{double} or \texttt{logical} array, preserving sparsity where possible, so that it may be used as a subscript in either a subscripted expression, like \texttt{arg(:,obj)}, or a subscripted assignment, like \texttt{lhs(:,obj) = rhs} where neither the \texttt{lhs} nor \texttt{rhs} argument need be a \texttt{sparse1} object.

Technical challenges of the \texttt{subsref} and \texttt{subsasgn} methods led to the \texttt{sparse1} class sometimes creating a temporary sparse \ac{2D} array. Unary method \texttt{subsref} implements a subscripted object reference, like \texttt{obj(:)} or \texttt{obj(:,j)} where \texttt{j} is an array of index values or a \texttt{logical} array that implies index values. Binary method \texttt{subsasgn} implements a subscripted assignment, like \texttt{lhs(:) = rhs} or \texttt{lhs(:,j) = rhs} where the \texttt{lhs} and/or \texttt{rhs} operands are \texttt{sparse1} objects. For efficiency, the design ensures the column dimension size of each such \ac{2D} array is equal to or linearly proportional to the number of nonzero columns.

Though one can convert a vector of linear index values, \texttt{lk}, into equally sized vectors, \texttt{i}, \texttt{j}, \texttt{k}, etc., of subscript index values, reorder them, and convert back to linear index values, the \texttt{permute} method favours compute-bound instead of such memory-bound steps to permute a sparse \ac{MDA}, which could reorder the dimension sizes. Meanwhile, keeping the linear index and nonzero values unchanged, along with the number of elements, the \texttt{reshape} method can, with almost zero cost, change the dimension sizes and number of dimensions.

The unary \texttt{sum} to \texttt{any} methods permute, reshape, and cast a \texttt{sparse1} object into \iac{2D} array, excising all-zero columns, and perform the required operation columnwise via the \texttt{sum} to \texttt{any} function. Such columnwise operation is most efficient considering MATLAB's internal sparse representation. Reversing the initial mapping from object to array yields, via a cast, permute, and reshape, a final \texttt{sparse1} object.

Methods \texttt{abs} to \texttt{round} apply functions \texttt{abs} to \texttt{round} to the internal sparse \ac{1D} array. Thus, MATLAB handles conversion from nonzeros to zeros, possible with \texttt{real}, \texttt{imag}, and \texttt{round}, and fill-in or conversion from zeros to nonzeros, possible with \texttt{log} and \texttt{not}. Fill-in aside, compute space and time are proportional to the initial number of nonzeros.

The \texttt{pagetranspose} and \texttt{pagectranspose} methods use \texttt{permute} to swap row and column subscripts, leaving all others unchanged. The latter conjugates nonzeros. For pagewise extraction of the main diagonal, \texttt{pagediag} maps the operation to and from \iac{3D} \texttt{sparse1} object and finds equal row and column subscripts. The \texttt{pagetrace} method operates similarly except that it performs a pagewise sum after the \texttt{pagediag} operation. Related \texttt{transpose}, \texttt{ctranspose}, \texttt{diag}, and \texttt{trace} methods invoke pagewise counterparts, for \ac{1D} and \ac{2D} objects, or throw an error, for \ac{MDA} objects.

In Table~\ref{tab:sparse1}, asterisked methods may accelerate, directly or indirectly, via the \texttt{spmex1} executable. Those accelerated directly, like \texttt{rref}, have \texttt{try}-\texttt{catch} blocks. The \texttt{try} block invokes \texttt{spmex1}. If the executable or required use case is missing, the \texttt{catch} block performs the operation via MATLAB functions. Methods that accelerate indirectly, like \texttt{null}, employ methods that accelerate directly, like \texttt{rref}.

Turning to binary operations after \texttt{subsasgn}, \texttt{sparse1} public methods fall into three patterns that Figure~\ref{fig:products} summarizes. Although it illustrates the \ac{BSX}, pagewise, and linear algebra designs for multiplication, the patterns apply to the \texttt{plus} to \texttt{complex}, the \texttt{pagemtimes} to \texttt{pagemrdivide}, and the \texttt{mtimes} to \texttt{mrdivide} methods, respectively.

\begin{figure}[t]
\mycaption{Patterns of key sparse \acs{2D} binary operations.}{(a) This Khatri-Rao or columnwise Kronecker product of sparse \acs{2D} arrays can compute \acs{BSX} products of sparse \acsp{MDA}. (b) Block-diagonal multiplication of sparse \acs{2D} arrays can compute pagewise multiplication of sparse \acsp{MDA}. (c) When matrix multiplication of sparse \acs{2D} arrays can represent a sparse \acs{MDA} operation, excising all-zero rows and columns in coordination improves efficiency.}
\label{fig:products}
\begin{align*}
\underbrace{\bracket{\mathbf{c}_1 & \cdots & \mathbf{c}_P}}_{\text{(a) } \mathbf{C}^{M N \times P}} &= \underbrace{\bracket{\mathbf{a}_1 & \cdots & \mathbf{a}_P}}_{\mathbf{A}^{M \times P}} \ast \underbrace{\bracket{\mathbf{b}_1 & \cdots & \mathbf{b}_P}}_{\mathbf{B}^{N \times P}} \\
&= \bracket{\mathbf{a}_1 \otimes \mathbf{b}_1 & \cdots & \mathbf{a}_P \otimes \mathbf{b}_P} \\[2ex]
\underbrace{\bracket{\mathbf{C}_1 & \cdots & \mathbf{0} \\ \vdots & \ddots & \vdots \\ \mathbf{0} & \cdots & \mathbf{C}_P}}_{\text{(b) } \mathbf{C}^{M P \times N P}} &= \underbrace{\bracket{\mathbf{A}_1 & \cdots & \mathbf{0} \\ \vdots & \ddots & \vdots \\ \mathbf{0} & \cdots & \mathbf{A}_P}}_{\mathbf{A}^{M P \times R P}} \underbrace{\bracket{\mathbf{B}_1 & \cdots & \mathbf{0} \\ \vdots & \ddots & \vdots \\ \mathbf{0} & \cdots & \mathbf{B}_P}}_{\mathbf{B}^{R P \times N P}} \\[2ex]
\underbrace{\bracket{\mathbf{C}_{11}^{\tilde{M} \times \tilde{N}} & \mathbf{0} \\ \mathbf{0} & \mathbf{0}}}_{\text{(c) } \mathbf{C}^{M \times N}} &= \underbrace{\bracket{\mathbf{A}_{11}^{\tilde{M} \times \tilde{P}} & \mathbf{0} \\ \mathbf{0} & \mathbf{0}}}_{\mathbf{A}^{M \times P}} \underbrace{\bracket{\mathbf{B}_{11}^{\tilde{P} \times \tilde{N}} & \mathbf{0} \\ \mathbf{0} & \mathbf{0}}}_{\mathbf{B}^{P \times N}}
\end{align*}
\vspace{-1ex}
\end{figure}

Like unary \texttt{rref} and \texttt{null} methods, binary \texttt{plus} (\verb|+|) to \texttt{complex} methods have \texttt{try}-\texttt{catch} paths. The \texttt{spmex1} executable, \ac{R2026}, accelerates operations only where operands and the result are real and sparse \texttt{double} arrays, meaning the \texttt{and} (\verb|&|) and \texttt{or} (\verb&|&) methods always go to the \texttt{catch} part. In the \texttt{catch} part, a \texttt{repmat} or \texttt{repelem} expression performs \iac{USX} of each operand. The \texttt{catch} part completes the \ac{BSX} operation by applying a MATLAB function, like \texttt{and} or \texttt{or}, to the expanded operands, both still sparse \ac{2D} arrays.

As relational operator methods \texttt{eq} (\texttt{==}) to \texttt{ge} (\texttt{>=}) require \texttt{logical} results, they first create a \texttt{sparse1} object, \texttt{obj}, equal to \texttt{arg1-arg2}. The subtraction invokes the \texttt{minus} method of \texttt{sparse1}, which \texttt{spmex1} would accelerate, because at least one operand, \texttt{arg1} or \texttt{arg2}, is a \texttt{sparse1} object here. The created object, \texttt{obj}, has an internal sparse \ac{1D} array of class \texttt{double}. Methods \texttt{eq} to \texttt{ge} replace it with the \texttt{logical} result of applying a MATLAB function, \texttt{eq} to \texttt{ge}, where the second argument is \texttt{0}, a \texttt{double} scalar. Unlike the initial numerical operation, the final relational operation may incur fill-in. When it happens, the fill-in is correct.

The \texttt{sparse1} class defines a \texttt{pagemtimes} method to support pagewise multiplication of two \acp{MDA}, one of which must be a \texttt{sparse1} object. Third and higher dimension sizes must equate or one operand must have only one page, in which case the operation proceeds with pagewise broadcasting of the \ac{2D} operand. For pagewise divisions, the class similarly defines \texttt{pagemldivide} and \texttt{pagemrdivide} methods.

In the main use case, \texttt{pagemtimes} operands have dimension sizes for matrix multiplication over rows and columns with a pagewise loop. Casting operands and result to block-diagonal forms, the \texttt{pagemtimes} method avoids loops. The \texttt{pagemldivide} and \texttt{pagemrdivide} methods operate similarly. Block-diagonal forms are \texttt{sparse1} \ac{2D} objects, not sparse \ac{2D} arrays. They may have much fewer nonzero rows and columns than row and column dimension sizes.

To perform the required operation on the \texttt{sparse1} \ac{2D} objects, the pagewise method invokes the \texttt{mtimes}, \texttt{mldivide}, or \texttt{mrdivide} method, all of which excise all-zero rows and all-zero columns consistently and outsource actual computation to MATLAB itself. The pagewise method maps the result, a \texttt{sparse1} \ac{2D} object, from its block-diagonal form to a \texttt{sparse1} \ac{MDA} object of appropriate dimension sizes.

Binary methods \texttt{mtimes} to \texttt{mrdivide} throw an error if either operand is not \ac{2D} or smaller. If either operand is scalar, \texttt{mtimes} returns a simpler \texttt{times} operation; likewise with scalar divisors for \texttt{mldivide} and \texttt{mrdivide}. Otherwise, dimension sizes must be compatible for the operation ($M \times P$ and $P \times N$ for \texttt{mtimes}). The methods disassemble each operand into row and column subscripts with nonzero values. After compressing row and column subscripts in coordination, the methods produce sparse \ac{2D} array operands. Applying a MATLAB \texttt{mtimes}, \texttt{mldivide}, or \texttt{mrdivide} yields a sparse \ac{2D} array (size $\tilde{M} \times \tilde{N}$ for \texttt{mtimes}) that disassembles into row and column subscripts, still compressed, with nonzero values. Decompressed row and column subscripts join nonzero values to construct a returned \texttt{sparse1} object having the correct dimension sizes (size $M \times N$ for \texttt{mtimes}).

To address how the \texttt{sparse1} class handles rank-deficient divisions, consider the \texttt{mtimes} operation shown in Figure~\ref{fig:products}. Assume, for simplicity, that all-zero rows of $\mathbf{A}$ and all-zero columns of $\mathbf{B}$ are at the bottom and right, respectively, and that all-zero columns of $\mathbf{A}$ that correspond to all-zero rows of $\mathbf{B}$ are at the right and bottom, respectively. The result, $\mathbf{C}$, always has all-zero rows and columns at the bottom and right. While dimension sizes, $\tilde{M}$, $\tilde{N}$, or $\tilde{P}$, of the ``top-left'' submatrices may not specify ranks, they predict a sparsity pattern in the result associated with sparsity patterns of operands.

The \texttt{sparse1} approach to rank-deficient divisions favours sparsity over unnecessary \texttt{NaN} fill-in, transferring only ``top-left'' submatrix divisions to MATLAB itself. Where \texttt{arg1} and \texttt{arg2} are sparse \ac{2D} arrays that represent ``top-left'' submatrices, the \texttt{mtimes} method wraps a MATLAB computation, \texttt{mtimes(arg1,arg2)}, with other MATLAB expressions to transform subscripts and nonzero values from one representation to another. Similarly, \texttt{mrdivide} and \texttt{mldivide} methods wrap MATLAB computations on ``top-left'' submatrices with pre- and post-processing of representations alone.

MATLAB supports methods, like \texttt{isequal}, and operator methods, like \texttt{vertcat} (\verb|[;]|), to specify true $N$-ary operations over multiple operands. When they come from multiple classes, MATLAB invokes the dominant class' method. A class may define inferior relationships to other classes, with built-in \texttt{double} and \texttt{logical} classes always inferior.

The \texttt{isequal} design models an approach to key $N$-ary challenges. The method must return \texttt{true} if all operands, \texttt{varargin\{:\}}, are equal and \texttt{false} otherwise. First, it uses a MATLAB function, \texttt{cellfun}, to compute a cell array, \texttt{dimsin}, having dimension sizes of each argument. Each must have a \texttt{size} method, perhaps built-in. The \texttt{isequal} method returns \texttt{false} if an $N$-ary MATLAB expression, \texttt{isequal(dimsin\{:\})}, is \texttt{false}. Otherwise, the method reshapes and casts each argument to a sparse column vector of built-in class, \texttt{double} or \texttt{logical} here, and returns an $N$-ary MATLAB expression, \texttt{isequal(varargin\{:\})}.

The \texttt{cat} design likewise exploits \texttt{cellfun}, $N$-ary MATLAB expressions, and sparse column vectors of built-in class. As a \texttt{horzcat} or \texttt{vertcat} invocation must match a \texttt{cat} invocation where the dimension of concatenation equals two or one, the \texttt{sparse1} class implements these operator methods exactly that way. In part because sparse column vectors of built-in class enable compute-bound operations, complexity of all \texttt{sparse1} $N$-ary methods depends linearly on the sum, given \texttt{sparse1} operands, of the numbers of nonzeros.

\subsection{Acceleration}

Table~\ref{tab:spmex1} summarizes use cases of the \texttt{spmex1} executable. Its source code comprises one C-file. This release involved a design comprising three subroutine groups: main, rearrange; and compute. There are no dependencies outside of the \ac{MEX} \ac{API}, included with the kernel, and the standard C library. Currently, the design supports only real \texttt{double} arrays.

\begin{table}[t]
\centering
\mycaption{Executable to accelerate key sparse \acs{MDA} operations.}{Coded in C, \texttt{spmex1} is a \acs{MEX} function designed mainly to accelerate sparse null-space calculation, via \acs{RRLU} factorization of one sparse \acs{2D} array, and to accelerate \acs{BSX} operations, via generalized Khatri-Rao products of two sparse \acs{2D} arrays.}
\label{tab:spmex1}
\vspace{1ex}
\begin{tabular}{p{3.3in}}
\hline
Use cases (\texttt{spmex1} executable) \\
\hline
\hspace{1em} \texttt{spmex1('debug',A)} displays internal details, following the C-language \acs{MEX} \acs{API}, of a MATLAB sparse \texttt{double} array, \texttt{A}.  Such arrays must be \acs{2D}, with size \texttt{[M N]}, or effectively \acs{1D}, if \texttt{M} or \texttt{N} is \texttt{1}. \\
\hspace{1em} \texttt{[L,U,p,q] = spmex1('rrlu',A)} returns a sparse \ac{LU} factorization with full pivoting, where \texttt{L*U} approximates \texttt{A(p,q)}, that also reveals rank. For a sparse \texttt{double} array, \texttt{A}, of size \texttt{[M N]}, lower- and upper-triangular sparse \texttt{double} arrays, \texttt{L} and \texttt{U}, have sizes \texttt{[M M]} and \texttt{[M N]}. The main diagonal of \texttt{L} is all one and the main diagonal of \texttt{U} is all nonzero up to the rank, \texttt{r}, of \texttt{A}. All rows of \texttt{U} below row \texttt{r} are zero. Row vectors \texttt{p} and \texttt{q} permute the sequences \texttt{1:M} and \texttt{1:N}. \\
\hspace{1em} \texttt{C = spmex1('bsx',fun,A,B)} returns an operation \texttt{fun}, either \texttt{'plus'} or \verb|'+'|, \texttt{'minus'} or \verb|'-'|, \texttt{'times'} or \verb|'.*'|, \texttt{'rdivide'} or \verb|'./'|, and \texttt{'ldivide'} or \verb|'.\'|, applied outerwise over rows and entrywise over columns of the operands, \texttt{A} and \texttt{B}. It realizes a broadcasted binary operation from two sparse \ac{2D} arrays to one sparse \ac{3D} array represented as a sparse \ac{2D} array. The output argument, \texttt{C}, has size \texttt{[M*N P]} if the input arguments, \texttt{A} and \texttt{B}, have sizes \texttt{[M P]} and \texttt{[N P]}. \\
\hline
\end{tabular}
\end{table}

Main subroutine \texttt{mexFunction} is the \texttt{spmex1} entrypoint. Using C-string comparisons, the subroutine determines applicability of the function to given arguments, issuing an error message and terminating if not applicable, and dispatching control to pairs of main subroutines, \texttt{mxCheckDebug} and \texttt{mxPrintDebug}, \texttt{mxCheckRRLU} and \texttt{mxCreateRRLU}, or \texttt{mxCheckBSX} and \texttt{mxCreateBSX}, based on use case.

The \texttt{mxCheckDebug} and \texttt{mxCheckRRLU} subroutines each verify that the second argument, \texttt{A}, is a sparse \texttt{double} array. With \texttt{mxCheckDebug}, there must be no output arguments. The \texttt{mxCheckBSX} subroutine verifies that the second to fourth arguments, \texttt{fun}, \texttt{A}, and \texttt{B}, are supported character vectors or sparse \texttt{double} arrays, following Table~\ref{tab:spmex1}. It also verifies that \texttt{A} and \texttt{B} have compatible dimension sizes for the \ac{BSX} operation, meaning the same number of columns.

The \texttt{mxPrintDebug} subroutine lists heap array entries of a given sparse \texttt{double} array, represented as an \texttt{mxArray} structure. A sparse \texttt{mxArray} stores row and column dimension sizes, \texttt{mrow} and \texttt{ncol}, pointers to heap arrays, \texttt{ir}, \texttt{nz}, and \texttt{jc}, and a maximum length, \texttt{nzm}, of two arrays, \texttt{ir} and \texttt{nz}. One integer array, \texttt{ir}, stores row index values contiguously by column. A floating-point array, \texttt{nz}, stores nonzero values contiguously by column. The third array, \texttt{jc}, has a length equal to \texttt{ncol} plus one, regardless of the number of nonzeros, \texttt{jc[ncol]}. This integer array, \texttt{jc}, stores a cumulative sum of the number of nonzeros per column.

Main subroutine \texttt{mxCreateRRLU} creates the four output arguments, \texttt{L}, \texttt{U}, \texttt{p}, and \texttt{q}, of the \texttt{'rrlu'} use case. It allocates heap memory for these arrays, as required. Consider the math model of \ac{LU} factorization, at first without pivoting, where matrices $\mathbf{L}$ and $\mathbf{U}$ correspond to \texttt{spmex1} arrays \texttt{L} and \texttt{U}:
\begin{align}
\underbrace{\bracket{1 & \mathbf{0} \\ \mathbf{l}_{21} & \mathbf{L}_{22}}}_{\mathbf{L}^{M \times M}} \underbrace{\bracket{a_{11} & \mathbf{a}_{12} \\ \mathbf{0} & \mathbf{U}_{22}}}_{\mathbf{U}^{M \times N}} &= \underbrace{\bracket{a_{11} & \mathbf{a}_{12} \\ \mathbf{a}_{21} & \mathbf{A}_{22}}}_{\mathbf{A}^{M \times N}} \text{.}
\end{align}
The first row of $\mathbf{U}$ equates to the first row of $\mathbf{A}$ and the first column of $\mathbf{L}$ depends piecewise on the first column of $\mathbf{A}$:
\begin{align}
\mathbf{l}_{21} &= \begin{cases}
\mathbf{a}_{21} / a_{11}, & a_{11} \ne 0 \text{,} \\
\mathbf{0}, & a_{11} = 0 \land \mathbf{a}_{21} = \mathbf{0} \text{.}
\end{cases}
\end{align}

Partial pivoting, required if $a_{11} = 0$ with $\mathbf{a}_{21} \ne \mathbf{0}$, swaps a lower row with the first row. A nonzero must exist somewhere in the first column. Considering round-off, partial pivoting done as a first step always ensures that, after pivoting, $a_{11}$ has the largest possible magnitude in the column. At some cost of complexity, full pivoting manages round-off even better. Full pivoting swaps a lower row with the first row and a right column with the first column so that, after pivoting, the pivot $a_{11}$ has the largest possible magnitude in the array.

Iterating the factorization with full pivoting yields submatrices, $\mathbf{L}_{22}$ and $\mathbf{U}_{22}$, from an updated submatrix, $\mathbf{A}_{22}'$:
\begin{align}
\mathbf{L}_{22} \mathbf{U}_{22} &= \underbrace{\mathbf{A}_{22} + \Delta\mathbf{A}_{22}}_{\mathbf{A}_{22}'} \text{,} \\
\Delta\mathbf{A}_{22} &= -\mathbf{l}_{21} \mathbf{a}_{12} \text{.}
\end{align}
The \texttt{spmex1} executable avoids recursion and does not create smaller matrices, represented by smaller \texttt{mxArrays}, either. Instead, it initializes and increments integer variables, \texttt{irow} and \texttt{jcol}, to indicate starting rows and columns of submatrices within preexisting, i.e., preallocated, \texttt{mxArrays}.

When computing a column of \texttt{U}, intermediate computations of columns to its right happens in preparation for subsequent final computation of columns to the right. To do this efficiently, three \texttt{mxArrays}, \texttt{U1}, \texttt{U2}, and \texttt{DA}, work with a boolean flag, \texttt{isU1}, that means \texttt{U1} is \texttt{U} when the flag is \texttt{true}, and \texttt{U2} is \texttt{U} otherwise. Thus, the \texttt{spmex1} design allocates two extra \texttt{mxArray}s that it frees before returning. Initially, \texttt{U1} and \texttt{U2} are created as duplicates of the given \texttt{mxArray}, \texttt{A}.

The \texttt{mxGetPiv} subroutine of the compute group does not alter a given sparse \texttt{mxArray}. It returns by reference the row and column index values, \texttt{i} and \texttt{j}, of the nonzero value having the largest magnitude. Where \texttt{irow} and \texttt{jcol} are input arguments, \texttt{mxGetPiv} ignores columns left of \texttt{jcol} and rows above \texttt{irow}. In the possible scenario where there is no nonzero meeting the requirements, \texttt{mxGetPiv} returns \texttt{false}. Otherwise, it returns \texttt{true}. The subroutine works columnwise, starting from column \texttt{jcol}. For each column examined, \texttt{mxGetPiv} skips rows above row \texttt{irow}.

Rearrange subroutines \texttt{mxMoveRow} and \texttt{mxMoveCol} perform full pivoting on a sparse \texttt{mxArray} passed by reference. Designed for submatrix operation, \texttt{mxMoveRow} ignores all columns left of column \texttt{j}, an argument. It moves row \texttt{i}, provided \texttt{i > j}, to row \texttt{j} and moves in-between rows down one row. This happens columnwise. Similarly, \texttt{mxMoveCol} ignores columns left of column \texttt{i}, where \texttt{i} is an argument. To insert column \texttt{j}, where \texttt{j > i}, at column \texttt{i} the in-between columns must be moved right. The subroutines move large contiguous \texttt{ir} and \texttt{nz} blocks with few standard C \texttt{memmove} and \texttt{memcopy} statements. Whereas \texttt{mxMoveRow} requires no alteration to the \texttt{jc} array, \texttt{mxMoveCol} performs simple updates to the number of nonzeros per column.

Though \texttt{mxMoveRow} copies single \texttt{ir} and \texttt{nz} values to stack memory temporarily, \texttt{mxMoveCol} requires enough heap memory, immediately after the actual number of nonzeros, to do so for an entire column. It invokes another rearrange subroutine, \texttt{mxSetNzmin}. This subroutine ensures that the storage available for nonzero and corresponding row index values is at least \texttt{nzmin}, an argument. If it is less, the available storage is doubled. Reallocation usually increases but may decrease the maximum length of the \texttt{ir} and \texttt{nz} arrays while also copying relevant old row index and nonzero values to the new arrays and, afterwards, freeing the old arrays.

Due to its initial value, \texttt{A}, the first row of \texttt{U}, initially \texttt{U1}, is correct. The first iteration of \texttt{mxCreateRRLU} computes the first column of \texttt{L} via the \texttt{mxSetLCol} subroutine of the compute group. The iteration computes and stores the submatrix $\Delta\mathbf{A}_{22}$ in \texttt{DA}, whose dimension sizes never change, via the \texttt{mxSetDelA} subroutine. Via the \texttt{mxUseDelA} subroutine, the iteration zeros the first column of \texttt{U2}, below the first row, and adds \texttt{U1} to \texttt{DA} at the submatrix level, storing the result in \texttt{U2}. The iteration completes by flipping \texttt{isU1} to \texttt{false}.

Because of pivoting of \texttt{U1}, entries in the working row, \texttt{irow}, have to be copied from \texttt{U1} to \texttt{U2} from the working column, \texttt{jcol}, and to the right. This is done columnwise along with additions of \texttt{U1} and \texttt{DA} from just below the working row and just right of the working column, \texttt{irow+1} and \texttt{jcol+1}, to the last row and last col. There is already memory, perhaps not enough, in the \texttt{ir} and \texttt{nz} arrays of \texttt{U2} for the results. Invocation of \texttt{mxSetNzmin}, with a simply calculated upper bound, guarantees there will be enough.

Second and subsequent iterations proceed similarly. The memory \texttt{DA} requires for the \texttt{jc} array never changes. Memory available for its \texttt{ir} and \texttt{nz} array may occassionally double through reallocation, a task managed by the \texttt{mxSetNzmin} subroutine. At the end of iteration \texttt{ncol}, \texttt{U1} is \texttt{U} if \texttt{isU1} equals \texttt{true} and \texttt{U2} is \texttt{U} otherwise. Iterations may terminate earlier via a \texttt{break} statement that actions if \texttt{mxGetPiv} returns \texttt{false}, meaning that the next submatrix, $\mathbf{A}_{22}'$, for which to compute an \ac{LU} factorization is all zero.

When pivoting, \texttt{mxMoveRow} has to be applied to \texttt{L} because of potential nonzero values below the working row to the left of the working column. A compute subroutine, \texttt{mxFinishL}, completes the main diagonal upon early termination.

To create initial \texttt{p} and \texttt{q} vectors, \texttt{mxCreateRRLU} invokes \texttt{mxCreateSeq} twice. This compute subroutine returns a \texttt{double} and dense \texttt{mxArray} having one row and \texttt{N} columns, with \texttt{N} an unsigned integer argument. Entries are integers from \texttt{1} to \texttt{N} in sequence. Each iteration, \texttt{mxCreateRRLU} records pivot choices by permuting \texttt{p} and \texttt{q} via two invocations of \texttt{mxMoveVal}. This rearrange subroutine processes a real and dense \texttt{mxArray} vector argument, like \texttt{p} or \texttt{q}. If \texttt{j > i}, where \texttt{i} and \texttt{j} are also arguments, the subroutine inserts the value from position \texttt{j} at position \texttt{i}, after moving in-between values to the right by one position to vacate space.

To realize the \texttt{'bsx'} use case in Table~\ref{tab:spmex1}, main subroutine \texttt{mxCreateBSX} invokes a compute subroutine, \texttt{mxPlusLike} or \texttt{mxTimesLike}, with pre- and post-processing. The use case follows MATLAB \ac{BSX} functionality of dense \texttt{double} arrays. Consider two \ac{3D} arrays, \texttt{A} and \texttt{B}, that have compatible sizes, \texttt{[M 1 P]} and \texttt{[1 N P]}, for \iac{BSX} operation, \texttt{bsxfun(fun,A,B)}, where \texttt{fun} specifies an operation like \texttt{@plus} or \texttt{@times}. The result dimension size, \texttt{[M N P]}, indicates the unary expansion or broadcasting of each operand at its singleton dimension, the second of \texttt{A} and the first of \texttt{B}, with an entrywise operation over the third dimension.

Given \texttt{permute} and \texttt{reshape}, a \ac{BSX} operation naturally requires \ac{3D} arrays but MATLAB does not have sparse \ac{3D} arrays. The \texttt{sparse1} class adds sparse \ac{MDA} objects to MATLAB. With it, operations like \texttt{A+B} or \texttt{A.*B}, where \texttt{A} and \texttt{B} have dimension sizes \texttt{[M 1 P]} and \texttt{[1 N P]}, become possible. Regardless of the original class of both operands, an initial cast ensures both going forward are \texttt{sparse1} objects, not just one as required. Internally, a \texttt{sparse1} object stores a sparse \ac{1D} array along with \ac{MDA} dimension sizes.

Using \texttt{permute}, \texttt{reshape}, \texttt{find}, \texttt{ind2sub}, \texttt{sparse}, and other functions, the \texttt{sparse1} class turns \ac{BSX} operands into sparse \ac{2D} arrays of sizes \texttt{[M leP]} and \texttt{[N leP]}, where \texttt{leP <= P}. Prior to creating the arrays, the class compresses subscripts in coordination for dimensions with common sizes, the reason for \texttt{leP}. As MATLAB demands space, and hence time, proportional to the number of columns, sparse \ac{2D} arrays must have the minimum such number for efficiency.

In the \texttt{sparse1} class, a \texttt{try}-\texttt{catch} approach offers two paths to realize all \ac{BSX} operations. An \texttt{spmex1} invocation performs the task in the \texttt{try} path, where possible, followed by fill-in as required. An easily computed scalar, the fill-in value equals the result of applying the relevant MATLAB function to a pair of zero scalars. Because vectorized MATLAB statements perform the fill-in, \texttt{mxCreateBSX} requires two subroutines as follows: \texttt{mxTimesLike}, where the nonzero pattern in the result equals at worst the intersection of the nonzero patterns of the broadcasted operands; and \texttt{mxPlusLike}, where the nonzero pattern in the result equals at worst the union of the nonzero patterns of the broadcasted operands.

In the \texttt{try} path, \texttt{spmex1} uses compute-bound C pointers and \texttt{for} loops in lieu of memory-bound \ac{USX} operations, what the \texttt{catch} path does with \texttt{repmat} and \texttt{repelem}. Moreover, \texttt{mxCreateBSX} calculates an upper bound for the number of nonzeros in the result and preallocates sufficient storage for it before invoking \texttt{mxTimesLike} or \texttt{mxPlusLike}, neither of which get involved in memory management. After \texttt{spmex1} invocation, checking the binary operation and comparing the number of nonzeros in the result to the number of elements determines, in the \texttt{try} path, the need for fill-in.

Regardless of which path the \ac{BSX} operation takes, \texttt{try} alone or a quick \texttt{try} followed by \texttt{catch}, the operation completes by producing and returning a \texttt{sparse1} object of the correct dimension sizes. A step before the operation begins precomputes these dimension sizes from the original \ac{MDA} sizes. Because of the \texttt{try}-\texttt{catch} approach, \texttt{sparse1} will work for all supported binary operations even when an end-user fails to compile the supplied \texttt{spmex1} C-file.


\section{Conclusions}
\label{sec:conclusions}

This paper identified two main approaches to numerical algebraic geometry. The first, continuation, does not concern tensor notation and/or algorithms. The second, Macaulay-\ac{CPD}, employs the \ac{CPD} formalism and algorithms to construct it, a tensor decomposition. Though tensors best associate with Ricci notation, \ac{CPD} literature uses an n-mode notation. Both numerical approaches solve for all solutions to a polynomial system without attempting to create a triangular system. Moreover, the \ac{CPD} one computes vectors that have to be invalidated as they are not actually part of the solution or zero set. The \ac{GB} approach to symbolic algebraic geometry, in contrast, produces a triangular-like system that ensures zero-set invariance.

Proving that a triangularization so defined guarantees zero-set invariance, this paper introduces the \ac{Qr} factorization using the \ac{RT} algebra, a tensor algebra with dual-variant index notation. Although it derives from the Ricci notation of differential geometry, previous disclosures about the \ac{RT} algebra concern tensors as numerical objects. The \ac{Qr} formalism, a generalization of the \ac{QR} formalism of linear algebra, generalizes unitary transforms to the algebraic geometry, possibly curved spaces, defined by zero sets of polynomial systems. This tensor theory, notationally compatible, must not be confused with the curved space theory of tensor calculus. Similar to the \ac{CPD} approach, a preliminary algorithm for \ac{Qr} factorization employs a null-space basis that, however, is not the null space of a Macaulay matrix but is a basis that guarantees triangularization.

More important than an iterative part of the proposed \ac{Qr} factorization, an optimization to minimize \iac{SSE} over unitary identities, this work contributed open-source tensor software, the RTToolbox, that can be used to investigate and develop improved algorithms. The paper elaborates on a new RTToolbox release, \ac{R2026}, having methods to construct the proposed \ac{Qr} and related formalisms. A sparse \ac{MDA} class, \texttt{sparse1}, collaborates with preexisting and augmented classes, \texttt{index} and \texttt{tensor}, to define sparse \texttt{tensor} objects, which enable a variety of unary, binary, and $N$-ary operations via a dual-variant \texttt{index} notation. Key operations, a unary null-space calculation and a variety of broadcasted binary operations, that were difficult to vectorize in MATLAB were instead dispatched to an open-source executable, \texttt{spmex1}, developed with pointers and loops in C. The class and the executable together represent and manipulate sparse \acp{MDA} using \ac{1D} and sometimes compressed \ac{2D} sparse arrays of MATLAB.

The paper includes two demos, a box dimensions and a robotic arm problem, chosen for their relative simplicity just to explain key concepts of the proposed \ac{RT} framework. For the box dimensions, the \ac{Qr} factorization yields numerically and without iteration what is effectively a reduced \ac{GB} triangularization. For the two-segment arm, given difficulties faced by the optimization part of the proposed algorithm, the demo includes a handcrafted pair of unitary transforms that prove the existence of a sparser \ac{Qr} factorization, which also satisfies required identities accurately. Thanks to the \texttt{sparse1} and \texttt{spmex1} designs, the paper demonstrates acceleration of key steps, like null space calculation, in comparison to an alternative path using MATLAB's own \texttt{null} function.


\section*{Acknowledgements}
\label{sec:acknowledgements}

The author gratefully acknowledges Adam P.\ Harrison and Dan Sirbu for helping get this research into a format suitable for sharing with university and industry representatives via the MathWorks Research Summit, 2026, held in Natick, MA.


\bibliography{RTFgeometry}
\bibliographystyle{spiejour}

\end{document}